\documentclass[reqno]{amsart}
\usepackage{amsthm}
\usepackage{amstext}
\usepackage{bbm}
\usepackage{enumitem}
\usepackage{amssymb}
\usepackage{quiver, adjustbox}
\usepackage{amsmath,calligra,mathrsfs}
\usepackage[all,cmtip]{xy}
\usepackage{dsfont}
\usepackage{hyperref}
\usepackage{graphicx}
\usepackage{float}
\usepackage{subfig}

\usepackage{colonequals}

\newcommand\bb{\mathbb}

\makeatletter
\renewcommand*\env@matrix[1][*\c@MaxMatrixCols c]{%
  \hskip -\arraycolsep
  \let\@ifnextchar\new@ifnextchar
  \array{#1}}
\makeatother

\theoremstyle{plain}
\newtheorem{thm}{Theorem}[section]
\newtheorem{lem}[thm]{Lemma}
\newtheorem{prop}[thm]{Proposition}
\newtheorem{cor}[thm]{Corollary}
\theoremstyle{definition}
\newtheorem{defn}[thm]{Definition}
\newtheorem{exmp}[thm]{Example}

\newtheorem{rem}[thm]{Remark}
\theoremstyle{plain}

\DeclareMathOperator{\sheafhom}{\mathscr{H}\text{\kern -3pt {\calligra\large om}}\,}

\DeclareMathOperator{\ECT}{ECT}

\DeclareMathOperator{\LECT}{LECT}
\DeclareMathOperator{\SELECT}{SELECT}
\DeclareMathOperator{\ERT}{ERT}

\usepackage{lipsum}
\usepackage{fullpage}
\usepackage{amsthm, thmtools}
\usepackage{parskip}

\usepackage{enumitem}

\renewenvironment{proof}{%
    \vspace{-\parskip}\begin{oldproof}%
    }{%
    \end{oldproof}\vspace{-\parskip}%
}

\let\phi\varphi
\usepackage{comment}
\usepackage{xcolor,mathtools}
\usepackage[square, numbers]{natbib}
\setcitestyle{round}

\title{Euler Characteristics and Homotopy Types of Definable Sublevel \\Sets, with Applications to Topological Data Analysis}

\author[1]{Mattie Ji$^*$}
\address{Department of Mathematics, University of Pennsylvania}
\email{mji13@sas.upenn.edu}

\author[2]{Kun Meng$^*$}
\address{Department of Statistics, Florida State University}
\email{kmeng@fsu.edu}

\begin{document}
\begin{abstract}
    Given a definable function $f: S \to \mathbb{R}$ on a definable set $S$, we study sublevel sets of the form $S^f_t \coloneqq \{x \in S: f(x) \leq t\}$ for all $t \in \mathbb{R}$. Using o-minimal structures, we prove that the Euler characteristic of $S^f_t$ is right-continuous with respect to $t$. Furthermore, when $S$ is compact, we show that $S^f_{t+\delta}$ deformation retracts to $S^f_t$ for all sufficiently small $\delta > 0$. Applying these results, we also characterize the connections between the following concepts in topological data analysis: the Euler characteristic transform (ECT), smooth ECT, Euler-Radon transform (ERT), and smooth ERT.
\end{abstract}


\maketitle

\section{Introduction}
\def\thefootnote{*}\footnotetext{Equal contributions.}
\noindent Sublevel sets have been widely used in both pure and applied branches of mathematics. Motivated by Morse theory and topological data analysis (TDA), we dedicate this article to exploring the Euler characteristics and homotopy properties of sublevel sets within the realm of tame topology \citep{van1998tame}. As an application to TDA, our results offer new perspectives and techniques for several topological descriptors of shapes and images that have been developed in the literature.

\subsection{Motivation I: Morse Theory} Informally, Morse theory \citep{milnor1963morse, matsumoto2002introduction} studies the topology of differentiable manifolds by analyzing critical points of a class of real-valued smooth functions known as Morse functions. Given a compact manifold $M$ and Morse function $f: M \to \bb R$, Morse theory is interested in sublevel sets of $M$ with respect to $f$, which we define as
\begin{align*}
    M_t^f :=\left\{x\in M: f(x) \le t\right\},\ \ \text{ for all }t \in \mathbb{R}.
\end{align*}
A classical result in Morse theory \citep[][Part I.3]{milnor1963morse} completely classifies the homotopy types of the collection $\{M_t^f: t \in \bb R\}$ based on the critical values of $f$ in $\bb R$. A consequence of this classification that we are interested in is the following \citep[][Remark 3.4, p. 20]{milnor1963morse}.
\begin{thm}\label{thm:morse_result}
    For all $\delta > 0$ sufficiently small, $M_{t+\delta}^f$ deformation retracts onto $M_t^f$. This also implies that the Euler characteristic of $M_{t+\delta}^f$ equals that of $M_t^f$ for all $\delta > 0$ sufficiently small.
\end{thm}
\noindent It is a corollary of Theorem \ref{thm:morse_result} that the function $t\mapsto \chi(M_t^f)$ is right-continuous, where $\chi(\cdot)$ denotes the Euler characteristic.

While Morse theory was originally developed by Marston Morse \citep{Morse1929} and was traditionally a subject in differential topology, it has since then inspired combinatorial adaptations such as discrete Morse theory \citep{Forman2002} and digital Morse theory \citep{Cox2003TopologicalZO} without necessarily requiring any smoothness. A natural question would then be - \textit{when $M$ is some ``shape" instead of a differential manifold and $f$ is no longer even smooth, how do the Euler characteristics and homotopy types of $M_t^f$ behave outside of a smooth category?} This is the question we will explore in this article.

\subsection{Motivation II: Topological Data Analysis}

Shape-valued data have emerged in various scientific domains. Traditionally, in applications, modeling shapes and evaluating (dis-)similarity between shapes have been achieved using either landmark-based or diffeomorphism-based methods \citep{kendall1989survey, dupuis1998variational, grenander1998computational, gao2019gaussianmorphometrics, gao2019gaussian}. However, these methods are not directly applicable to many databases used in applications, as extensively discussed in the literature \citep[e.g.,][]{turner2014persistent, wang2021statistical}. TDA offers innovative approaches to modeling shapes, $S\subseteq\mathbb{R}^n$, without reliance on landmarks or diffeomorphisms \citep{turner2014persistent}. One prominent source of inspiration in TDA is Morse theory---one common practice is to look at topological invariants (e.g., Euler characteristics) of sublevel sets of $S$ with respect to $f: S\rightarrow\mathbb{R}$, which we define as 
\begin{align*}
    S_t^f :=\left\{x\in S: f(x) \le t\right\},\ \ \text{ for all }t \in \mathbb{R}.
\end{align*}
The following special case is of particular importance in TDA: $f(x) = \phi_v(x):= x \cdot v$ and $v\in\mathbb{S}^{n-1}$ is a fixed unit vector. This special case is a building block of the Euler characteristic transform \citep[ECT, ][]{turner2014persistent}. The ECT and related integral transforms are of interest to many topological data analysts and have been widely utilized in applied sciences \citep[e.g.,][]{crawford2020predicting, wang2021statistical, marsh2022detecting, meng2022randomness}. \cite{munch2023invitation} provided a comprehensive survey of the ECT from both theoretical and applied perspectives. The ECT represents shapes utilizing integer-valued functions. Specifically, for a given shape $S$ in $\mathbb{R}^n$, its ECT is defined as the function $\operatorname{ECT}(S):\,(v,t) \mapsto \operatorname{ECT}(S)(v,t):=\chi(S_t^v)$, where $\chi(\cdot)$ denotes the Euler characteristic and $S_t^v:=\left\{x\in S: x\cdot v \le  t\right\}$ for $(v, t) \in \mathbb{S}^{n-1}\times\mathbb{R}$. See Figure \ref{fig:Fig1a} for an illustration. Notably, \cite{ghrist2018persistent} and \cite{curry2022many} conclusively demonstrated that the descriptor $\operatorname{ECT}(S)$ preserves all the information within the shape $S$ when $S$ is compact. Precisely, the shape-to-ECT map $S\mapsto\operatorname{ECT}(S)$ is injective on compact definable sets.

To incorporate the techniques in functional data analysis \citep{hsing2015theoretical} and Gaussian process regression \citep{williams2006gaussian}, \cite{crawford2020predicting} introduced the smooth ECT (SECT) by smoothing the ECT via Lebesgue integration. Precisely, given a shape $S\subseteq\mathbb{R}^n$ bounded by an open ball $B(0,R):=\{x\in\mathbb{R}^n: \Vert x\Vert<R\}$, the SECT of $S$ is defined as $\operatorname{SECT}(S):=\{\operatorname{SECT}(S)(v, t): (v,t)\in\mathbb{S}^{n-1}\times\mathbb{R}\}$, where
\begin{align}\label{eq: def of SECT(S)(v,t)}
    \operatorname{SECT}(S)(v,t):=\int_{-R}^t \chi(S_\tau^v)\,d\tau - \frac{t+R}{2R}\int_{-R}^R \chi(S_\tau^v)\,d\tau,
\end{align}
for all $(v, t) \in \mathbb{S}^{n-1}\times [-R,\,R]$. The $\operatorname{SECT}(S)(v,t)$ defined in Equation~\eqref{eq: def of SECT(S)(v,t)} can be viewed as an analog of a Brownian bridge \citep[][Chapter 21]{klenke2013probability} if we view $\int_{-R}^t \chi(S_\tau^v)\,d\tau$ as an analog of a Brownian motion. The SECT converts shape-valued data (e.g., tumors and mandibular molars of primates) into functional data, which was particularly employed by \cite{meng2022randomness} to perform hypothesis testing on shapes via the analysis of variance for functional data \citep{gorecki2015comparison}. The SECT has been extensively applied in sciences, e.g., organoid morphology \citep{marsh2022detecting}, radiomics \citep{crawford2020predicting}, and evolutionary biology \citep{meng2022randomness}.

The SECT is formulated from the ECT using Lebesgue integrals, as demonstrated in Equation~\eqref{eq: def of SECT(S)(v,t)}. Now, if we are given the SECT of a shape $S$ and want to recover the corresponding $\ECT(S)$, challenges arise due to the nature of Lebesgue integration. Without additional regularity properties (e.g., right continuity) of the function $t\mapsto\operatorname{ECT}(S)(v,t)=\chi(S_t^v)$, one can recover the $\operatorname{ECT}(S)(v,t)$ from the $\operatorname{SECT}(S)(v,t)$ only in the sense of ``almost everywhere for $t$ with respect to Lebesgue measure" rather than ``exactly everywhere." To put it more concretely, given $\{\operatorname{SECT}(S)(v,t)\}_{t\in\mathbb{R}}$ without the knowledge of the regularity of $t\mapsto\operatorname{ECT}(S)(v,t)$, one can only determine the values $\{\operatorname{ECT}(S)(v,t):t\in\mathbb{R}-N\}$ through the Radon-Nikodym derivative with respect to the Lebesgue measure, where $N$ is a measurable subset of $\mathbb{R}$ with Lebesgue measure zero. To recover the values of $\operatorname{ECT}(S)(v,t)$ for $t\in N$, we need the right continuity $t\mapsto\operatorname{ECT}(S)(v,t)=\chi(S_t^v)$, which is analogous to Theorem \ref{thm:morse_result} and is one of the contributions of this article.

Each shape $S$ can be equivalently identified as the indicator function $\mathbbm{1}_S$ of $S$. With this perspective, the ECT can be generalized to take suitable real-valued functions rather than a given shape $S$ (equivalently $\mathbbm{1}_S$). In this paper, we study the following two generalizations of the ECT: 
\begin{enumerate}
    \item \cite{kirveslahti2023representing} proposed the lifted and super lifted Euler characteristic transform (LECT and SELECT) based on computing the Euler characteristics of level sets and superlevel sets, respectively. 
    \item \cite{meng2023Inference} proposed the Euler-Radon transform (ERT) using the Euler integration framework developed by \cite{baryshnikov2010euler} to model grayscale images in medical imaging. Given a suitable compactly supported real-valued function $g: S \subseteq \bb R^n \to \bb R$, which represents a grayscale image, the ERT of $g$ is a function $(v,t) \mapsto \ERT(g)(v,t)$ defined for all $(v,t)\in\bb S^{n-1} \times \bb R$.
\end{enumerate}
The precise definitions of the topological descriptors mentioned above will be provided in Section~\ref{section: Background: O-minimal Structures and Euler Calculus}. An important contribution of this article is to connect these topological descriptors.

\subsection{Overview of Contributions and Article Organization}

As an analog to Theorem~\ref{thm:morse_result} in Morse theory, we prove the following main results in this paper.
\begin{enumerate}
    \item Given a ``definable" shape $S \subseteq \bb R^n$ and a ``definable" function $f: S \to \bb R$, we show in Theorem~\ref{thm:general_right_continuous} that the map $t \mapsto \chi(S^f_t)$ is right-continuous. The notion of ``definability" is a fundamental concept in tame topology \citep{van1998tame} and will be precisely defined in Section~\ref{section: Background: O-minimal Structures and Euler Calculus}. 
    \item Furthermore, if the shape $S$ is compact and $f$ is continuous, we show in Theorem~\ref{thm:homotopy_K} that $S^f_{t + \delta}$ deformation retracts onto $S^f_{t}$ for sufficiently small $\delta > 0$, which is analogous to Theorem \ref{thm:morse_result}.
    \item Using the results presented above, we show in Theorem~\ref{thm: ERT vs. SERT} that the ERT can be recovered from the smooth Euler-Radon transform \citep[SERT,][]{meng2023Inference}. As a corollary of Theorem~\ref{thm: ERT vs. SERT}, one can recover the ECT from the SECT. Additionally, in Corollary \ref{cor: connect the ECT, LECT, and SELECT}, we provide a formula that connects the LECT, SELECT, and ECT.
\end{enumerate}

Importantly, the fact that the ECT can be recovered from the SECT establishes a bridge that connects algebraic topology to functional analysis, potentially inviting probability theory on Banach spaces to the realm of algebraic topology. Specifically, \cite{ghrist2018persistent} and \cite{curry2022many} have shown that a shape $S$ can be recovered from its $\mathrm{ECT}(S)$. Consequently, by our Theorem~\ref{thm: ERT vs. SERT}, a shape $S$---a topological object---can also be recovered from its $\mathrm{SECT}(S)$. Moreover, \cite{meng2022randomness} show that $\mathrm{SECT}(S)$ belongs to the separable Banach space $C(\mathbb{S}^{n-1};\mathcal{H})$, the collection of continuous maps from the unit sphere $\mathbb{S}^{n-1}$ to a reproducing kernel Hilbert space $\mathcal{H}$. Thus, each shape $S$ can be identified as a point, i.e., $\mathrm{SECT}(S)$, in the Banach space $C(\mathbb{S}^{n-1};\mathcal{H})$, which makes it possible to study shapes using tools from functional analysis. For instance, since probability theory on separable Banach spaces is well developed \citep{hairer2009introduction}, one may use probability measures on $C(\mathbb{S}^{n-1};\mathcal{H})$ to analyze the randomness of the shape $S$.

The remainder of this paper is organized as follows. In Section~\ref{section: Background: O-minimal Structures and Euler Calculus}, we review the background and context for o-minimal structures, Euler calculus, ECT, LECT, SELECT, and ERT. In Section~\ref{section: A General Right Continuity Result}, we prove in Theorem~\ref{thm:general_right_continuous} the right continuity of the map $t \mapsto \chi(S^f_t)$ for definable sets $S$ and definable functions $f: S \to \bb R$. In Section~\ref{section: Right Continuity of the Euler Characteristic Transform} and Section~\ref{section: Right Continuity of the Euler-Radon Transform}, we discuss its applications to showing that $t \mapsto \operatorname{ECT}(S)(v,t)$ and $t \mapsto \operatorname{ERT}(g)(v,t)$ are both right-continuous for each fixed $v \in \bb S^{n-1}$ in Theorem~\ref{thm:ECT_right_continuous} and Theorem~\ref{thm:ERT_right_continuous} respectively. In Section~\ref{section: Middle Continuity of the Euler Characteristic}, we will discuss an application of Theorem~\ref{thm:general_right_continuous} in proving a ``middle continuity" result for the Euler characteristic. In Section~\ref{section: Right Continuity of Homotopy Type}, we prove in Theorem~\ref{thm:homotopy_K} that for a compact definable set $K$, $K^f_{t+\delta}$ deformation retracts onto $K^f_t$ for all $\delta > 0$ sufficiently small. In Section~\ref{section: Corollaries of Right Homotopy}, we discuss corollaries of Theorem~\ref{thm:homotopy_K}, including a ``middle continuity" result for homotopy type. As an application of Sections \ref{section: Right continuity} and \ref{section: Homotopy Types}, we also characterize the connections between the ECT, SECT, ERT, and SERT in Section~\ref{section: Applications to the Smooth Euler Characteristic Transform}.

\section{Background}\label{section: Background: O-minimal Structures and Euler Calculus}

In this section, we briefly cover the necessary background in o-minimal structures, Euler calculus, the ECT, and two relevant extensions of the ECT to real-valued definable functions. We refer the reader to \cite{van1998tame} for more details on o-minimal structures, to \cite{curry2012euler} for more details on Euler calculus, and to \cite{ghrist2018persistent} and \cite{curry2022many} for more details on the ECT.

\subsection{O-minimal Structures} 

The goal of o-minimal structures is to create a collection of subsets of Euclidean spaces that abstracts the features of ``well-behaved sets" such as the semialgebraic and semilinear sets \cite[][Chapters 1 and 2]{van1998tame}, while excluding ``poorly-behaved sets" like the topologist's sine curve given by the graph of $x\mapsto \sin(\frac{1}{x})$ in $\bb R^2$. O-minimal structures are defined as follows:
\begin{defn}\label{def:o-minimal}
Suppose we have a sequence $\mathcal{O}=\{\mathcal{O}_n\}_{n \geq 1}$ where $\mathcal{O}_n$ is a Boolean algebra of subsets of $\mathbb{R}^n$ for each $n$. We call $\mathcal{O}$ an \textit{o-minimal structure} on $\mathbb{R}$ if it satisfies the following:
    \begin{enumerate}
        \item If $A \in \mathcal{O}_n$, then $A \times \bb R\in \mathcal{O}_{n+1}$ and $\bb R \times A \in \mathcal{O}_{n+1}$.
        \item $\{(x_1, ..., x_n) \in \bb R^n\ |\ x_i = x_j\} \in \mathcal{O}_n$ for $1 \leq i < j \leq n$.
        \item $\mathcal{O}$ is closed under axis-aligned projections.
        \item $\{r\} \in \mathcal{O}_1$ for all $r \in \bb R$ and $\{(x, y) \in \bb R^2\ |\ x < y\} \in \mathcal{O}_2$.
        \item \label{cond:o1-definable} $\mathcal{O}_1$ is exactly the finite unions of points and open intervals.
        \item $\mathcal{O}_3$ contains the graphs of addition and multiplication.
    \end{enumerate}
\end{defn}
\noindent In Definition \ref{def:o-minimal}, Conditions (1)-(5) form the fundamental definition of o-minimal structures as presented in \cite{van1998tame}. To utilize powerful theorems, such as the ``trivialization theorem" referenced in our Theorem~\ref{thm:general_right_continuous}, one also requires Condition (6). Many authors choose to define o-minimal structures on $\bb R$ to include Condition (6), e.g., \cite{curry2012euler} and \cite{Coste2002ANIT}. A notable consequence of assuming Conditions (1)-(6), due to the Tarski-Seidenberg theorem \citep{TarskiMcKinsey+1951}, is that any o-minimal structures on $\bb R$ defined this way must encompass the collection of all semialgebraic sets \citep[][Example 1.2]{dries_real}. Definition~\ref{def:o-minimal} is also sometimes called an o-minimal expansion of the real numbers.

The concepts in the following definition will be employed in our study. They are commonly used in the theoretical TDA literature \citep[e.g.,][]{ghrist2018persistent, curry2022many, kirveslahti2023representing}. More details regarding these concepts are available in \cite{van1998tame} and \cite{baryshnikov2010euler}.
\begin{defn} 
Suppose an o-minimal structure $\mathcal{O}=\{\mathcal{O}_n\}_{n\ge1}$ on $\mathbb{R}$ is given.
\begin{enumerate}
    \item A subset $S$ of $\bb R^n$ is called a \textit{definable} set if $S \in \mathcal{O}_n$. Throughout this paper, a definable set is also referred to as a \textit{shape}.
    \item Let $X$ be definable. A function $f: X \to \bb R^n$, for some $n$, is said to be \textit{definable} if its graph is definable.
    \item A function that is both continuous and definable, and possesses a continuous, definable inverse, is called a \textit{definable homeomorphism}. Two definable sets are \textit{definably homeomorphic} if there is a definable homeomorphism between them.
    \item A definable function is called \textit{constructible} if it is integer-valued. Without loss of generality, we will restrict the codomain of constructible functions to $\bb Z$.
\end{enumerate}
\end{defn}
\noindent Note that every constructible function is bounded. This is because the image of a constructible function is a definable subset of $\bb Z$, which is a finite union of points in $\bb Z$ by Condition~(\ref{cond:o1-definable}) of Definition~\ref{def:o-minimal}.

\subsection{Euler Calculus} The Euler calculus is based on the observation that the Euler characteristic is finitely additive and well-defined for certain well-behaved subsets of $\bb R^n$. The main theme of Euler calculus is to apply the Euler characteristic as an analog of a signed measure. The subject was originally introduced by \cite{Schapira_1991, schapira1995tomography} and \cite{Viro_1988}.

By the ``cell decomposition theorem" \citep[][Chapter 3, Theorem 2.11]{van1998tame}, any definable set $S$ can be partitioned into cells $S_1, ..., S_N$. The (combinatorial) \textit{Euler characteristic} of $S$ is defined as
\begin{align}\label{eq: def of EC}
    \chi(S) = \sum_{i = 1}^N (-1)^{\dim(S_i)},
\end{align}
where $\dim(S_i)$ denotes the dimension of the cell $S_i$ \citep[][Section 1 of Chapter 4 therein, for a precise definition of dimensions]{van1998tame}. Proposition 2.2 from Chapter 4 of \cite{van1998tame} shows that the value of $\chi(S)$ is independent of the choice of cell decomposition. On a locally compact definable set $K$, the Euler characteristic $\chi(\cdot)$ defined in Equation~\eqref{eq: def of EC} is equivalent to the alternating sum of Betti numbers via the Borel-Moore homology or cohomology with compact support \citep[][Lemma 8.5]{curry2012euler}. That is, $\chi(K)=\sum_{n\in\mathbb{Z}}(-1)^n \cdot \operatorname{dim} H_n^{BM}(K;\mathbb{R})$, where $H_*^{BM}$ denoting the Borel-Moore homology \citep{bredon2012sheaf}. Over compact definable sets, $\chi(S)$ is equal to the alternating sum of Betti numbers from the singular homology. Notably, $\chi(S)$ is a homotopy invariant if $S$ is compact but is only a definable homeomorphism invariant in general.

With the Euler characteristic $\chi(\cdot)$ in Equation~\eqref{eq: def of EC}, the Euler integration functional $\int(\cdot)d\chi$ is defined as follows \citep[also see][Section 3.6]{ghrist2014elementary}.
\begin{defn}
   For any constructible function $g: X \to \bb Z$, we define its \textit{Euler integral} as
   \begin{equation}\label{eq:euler-integral}
      \int_X g \,d\chi \coloneqq \sum_{n=-\infty}^{+\infty} n \cdot \chi\left(g^{-1}(n)\right). 
   \end{equation}
   Since $g$ is constructible, it is bounded, and each $g^{-1}(n)$ is definable. Therefore, Equation~(\ref{eq:euler-integral}) is well-defined.
\end{defn}

It is worth remarking that there is also a different approach to Euler calculus and definable sets using the constructible sheaves. This approach was taken in Schapira's original paper on this topic \citep{Schapira_1991}. We refer the reader to \cite{Kashiwara_Schapira_1990} for a thorough introduction to constructible sheaves.

\subsection{Euler Characteristic Transform} Hereafter, for any subset $S \subseteq \bb R^n$ and function $f: S \to \bb R$, we adopt the following notations for sub-level sets,
\begin{align}\label{eq: sublevel set notation}
    S^{f}_{t} \coloneqq \{x \in S\ |\ f(x) \leq t\},
\end{align}
for all $t \in \bb R$. In the special case where $f(x) = \varphi_v(x) \coloneqq x \cdot v$ for a fixed direction $v \in \bb S^{n-1}$, we write $S^{\varphi_v}_t$ in the following notation instead
\begin{align*}
    & S^{v}_{t} \coloneqq \{x \in S\ |\ x \cdot v \leq t \}.
\end{align*}
Figure \ref{fig:Fig1a} illustrates an example of $S^f_t$ with $S = \{(x, y) \in \mathbb{R}^2\ |\ 1 \leq x^2 + y^2 \leq 4\}$ and \(f(x) = \varphi_v(x) \coloneqq x \cdot v\). Figure \ref{fig:Fig1b} illustrates another example of $S^f_t$ with the same $S$ as in Figure~\ref{fig:Fig1a} and $f(x)=x^2-y^2$.

If $S$ and $f$ are both definable, then $S^f_t$ is definable since our o-minimal structure includes all real semialgebraic sets and is closed under finite intersections.

\begin{figure}[h]
  \centering
  \subfloat[$S^{f}_{t}$ with $f(x) = x \cdot v$.]{\includegraphics[width=0.38\textwidth]{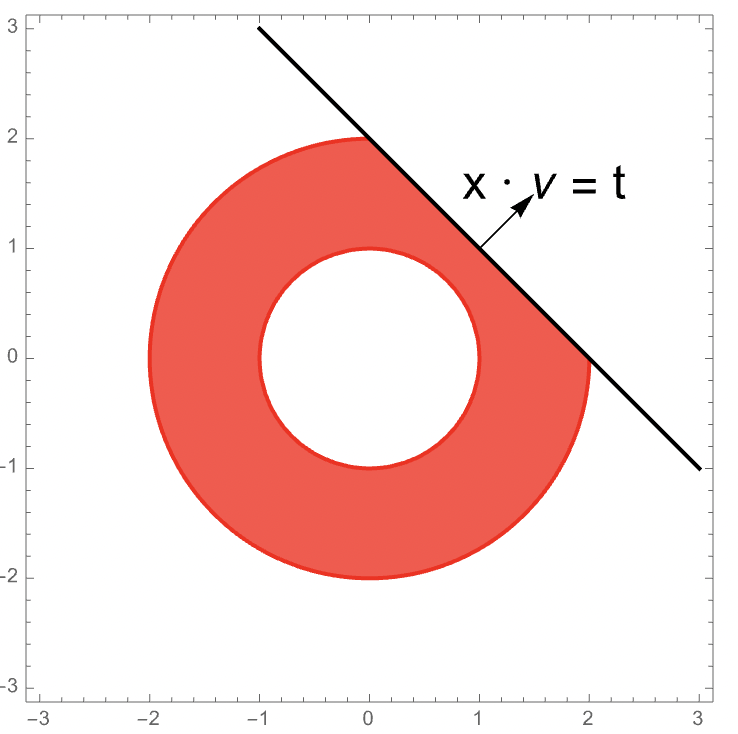}\label{fig:Fig1a}}
  \hfill
  \subfloat[$S^{f}_{t}$ with $f(x) = x^2 - y^2.$]{\includegraphics[width=0.42\textwidth]{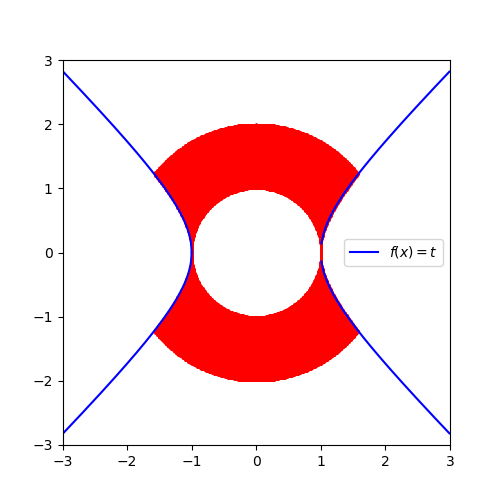}\label{fig:Fig1b}}
  \caption{Illustrations of the sublevel sets of $S$, an annulus in $\mathbb{R}^2$. In each panel, the sublevel set $S^{f}_{t} \coloneqq \{x \in S\ |\ f(x) \leq t\}$ is illustrated by the (red) solid region.}
  \label{fig:sublevel_set}
\end{figure}

The ECT was introduced by \cite{turner2014persistent}. Using the definition of Euler integration from Equation~\eqref{eq:euler-integral}, \cite{ghrist2018persistent} and \cite{curry2022many} provided a representation and generalization of the ECT as follows.
\begin{defn}[Euler Characteristic Transform]\label{def:ECT}
    (i) Let $X \subseteq \bb R^n$ be a definable set and $g: X \to \bb Z$ be a constructible function. Then, the \textit{Euler characteristic transform} of $g$ is a function $(v, t) \mapsto \operatorname{ECT}(g)(v, t)$ on $\bb S^{n-1} \times \bb R$ defined as follows
    \begin{align}\label{eq: ECT of CF}
        \ECT(g)(v, t) \coloneqq \int_X g \cdot \mathbbm{1}_{X^{v}_t}\  d\chi.
    \end{align}
    (ii) In the special case that $X = \bb R^n$ and $g$ is the indicator function of a definable subset $S \subseteq \bb R^n$, we write $\ECT(S) \coloneqq \ECT(\mathbbm{1}\{S\})$, i.e., 
    \begin{align}\label{eq: ECT of CS}
        \operatorname{ECT}(S)(v,t)=\chi(S_t^v),\ \ \ \text{ for all }(v,t)\in\mathbb{S}^{n-1}\times\mathbb{R}.
    \end{align}
\end{defn}
\noindent Equation~\eqref{eq: ECT of CS} represents the version of the ECT as originally proposed in \cite{turner2014persistent}. Meanwhile, \cite{jiang2020weighted} utilized the version depicted in Equation~\eqref{eq: ECT of CF} to analyze images of glioblastoma multiforme tumors.

Finally, we will need the following lemma, which is a trivial generalization of Lemma 3.4 of \cite{curry2022many}:
\begin{lem}\label{lem:general_tame_property}
    Suppose $S \subseteq \bb R^n$ is a definable set and $f: S \to \bb R$ is a definable function.
    \begin{enumerate}
        \item $S^f_t$ falls into finitely many homeomorphism types as $t$ ranges over $\bb R$.
        \item $\chi(S^f_t)$ takes finitely many values as $t$ ranges over $\bb R$.
        \item The function $t \mapsto \chi(S^f_t)$ has at most finitely many discontinuities; the function is constant between any two consecutive distinct discontinuities.
    \end{enumerate}
\end{lem}
\noindent Lemma~\ref{lem:general_tame_property} follows as a consequence of the ``trivialization theorem" \citep[][Chapter 9, Theorem 1.2]{van1998tame}. More specifically, it requires the following lemma from Chapter 9 of \cite{van1998tame}.

\begin{lem}[Rephrased from $\S$2 of Chapter 9 of \cite{van1998tame}]\label{lemma:: finite homeomorphism types}
Let $X \subseteq \mathbb{R}^{m+n}$ be a definable set. For any $t\in\mathbb{R}^m$, define $X_t :=\{x\in\mathbb{R}^n\vert\, (t,x)\in X\}$. Then, there exists a finite definable partition $\{A_i\}_{i=1}^M$ of $\mathbb{R}^m$, together with definable sets $\{F_i\}_{i=1}^M\subseteq\mathbb{R}^N$ for some $N$, such that $X_t$ is definably homeomorphic to $F_i$ for all $t \in A_i$.
\end{lem}

\begin{proof}[Proof of Lemma~\ref{lem:general_tame_property}]
We implement Lemma \ref{lemma:: finite homeomorphism types} by defining the following: (i) $m := 1$ and (ii) $X := \left\{(t,x)\in \mathbb{R}\times\mathbb{R}^n\,\vert\, x\in S \text{ and } f(x) - t \le 0\right\}$. Note that $X$ is definable. Lemma \ref{lemma:: finite homeomorphism types} implies that $X_t=S_t^f$ falls into at most $M$ homeomorphism types as $t$ ranges over $\bb R$. Since the Euler characteristic is a definable homeomorphism invariant, clearly $\chi(S^f_t)$ can only take at most $M$ values as $t$ runs through $\bb R$.

The discussion above shows that the function $t \mapsto \chi(S_t^f)$ is a definable function. The ``cell decomposition theorem" \citep[][Chapter 3, Theorem 2.11]{van1998tame} implies that $\bb R$ has a finite partition into cells such that the function $t \mapsto \chi(S_t^f)$ is continuous on each cell. Therefore, the function $t \mapsto \chi(S_t^f)$ has at most finitely many discontinuities, and the function is constant between any two consecutive distinct discontinuities. 
\end{proof}

\subsection{Extending the ECT to Real Definable Functions}

One limitation of the ECT is that it can only take in integer-valued functions and does not apply to most real-valued functions. Several papers have discussed possible generalizations of the ECT, and here we briefly outline two approaches.

\subsubsection{The Lifted and Super Lifted Euler Characteristic Transform}

Motivated by Gaussian random fields, \cite{kirveslahti2023representing} introduced the lifted and super lifted Euler characteristic transform (LECT and SELECT) to capture the Euler characteristics of level sets and superlevel sets of a definable function.

\begin{defn}\label{def:lect_select}
    Let $X \subseteq \bb R^n$ be a definable set and $g: X \to \bb R$ be a definable function, then the \textit{lifted Euler characteristic transform} of $g$ is a function $(v, t, s) \mapsto \LECT(g)(v, t, s)$ on $\bb S^{n-1} \times \bb R \times \bb R$ defined as follows
    \[\LECT(g)(v, t, s) \coloneqq \int_X \mathbbm{1}_{\{g = s\}} \cdot \mathbbm{1}_{X^v_t} \,d\chi.  \]
    Similarly, the \textit{super lifted Euler characteristic transform} of $g$ is a function $(v, t, s) \mapsto \SELECT(g)(v, t, s)$ on $\bb S^{n-1} \times \bb R \times \bb R$ defined as follows
    \[\SELECT(g)(v, t, s) \coloneqq \int_X \mathbbm{1}_{\{g \geq s\}} \cdot \mathbbm{1}_{X^v_t} \,d\chi.  \]
\end{defn}
\noindent Regarding the LECT and SELECT, we have the following lemma that is similar to the case of Lemma~\ref{lem:general_tame_property}.
\begin{lem}\label{lem:SELECT_LECT_tame_property}
  Suppose $S \subseteq \bb R^n$ is a definable set and $g: S \to \bb R$ is a definable function. Then
  \begin{enumerate}
      \item $\LECT(g)(v, t, s)$ and $\SELECT(g)(v, t, s)$ take only finitely many values as $(v, t, s)$ runs through $\bb S^{n-1} \times \bb R \times \bb R$.
      \item The functions $t \to \LECT(g)(v, t, s)$ and $t \to \SELECT(g)(v, t, s)$ have at most finitely many discontinuities.
  \end{enumerate}
\end{lem}
\noindent Lemma \ref{lem:SELECT_LECT_tame_property} is adapted from \cite{meng2023Inference}. The proof is similar to that of Lemma~\ref{lem:general_tame_property} and is omitted here.

\subsubsection{The Euler-Radon Transform}

\cite{meng2023Inference} introduced the Euler-Radon transform (ERT) based on the framework of Euler integration for real definable functions proposed by \cite{baryshnikov2010euler}. Similar to approximating real integrals using a Riemann sum, the idea proposed by \cite{baryshnikov2010euler} was to integrate real definable functions with approximations by the floor and ceiling functions (denoted by $\lfloor \cdot \rfloor$ and $\lceil \cdot \rceil$, respectively). More precisely, given a compactly supported definable function $g: X \to \bb R$, we adopt the following the Euler integral of $g$ 
\begin{align}\label{eq: Euler integrals, averaged version}
\int_X g \,[d\chi] \coloneqq \lim_{n \to \infty} \left( \frac{1}{2n} \int_X \lfloor ng \rfloor + \lceil ng \rceil \,d\chi \right).
\end{align}
\cite{baryshnikov2010euler} showed that the limit in Equation~\eqref{eq: Euler integrals, averaged version} exists and is well-defined. \cite{meng2023Inference} defined the Euler-Radon transform using the functional $\int (\cdot) \,[d\chi]$ as follows
\begin{defn}\label{def:ert}
    Let $X \subseteq \bb R^n$ be a definable set and $g: X \to \bb R$ be a compactly supported definable function. The \textit{Euler-Radon transform} of $g$ is a function $(v, t) \mapsto \ERT(g)(v, t)$ on $\bb S^{n-1} \times \bb R$ defined as follows
    \[\ERT(g)(v, t) \coloneqq \int_X g \cdot \mathbbm{1}_{X^v_t} [d\chi].\]
\end{defn}
\noindent Note that when $g$ is a constructible function, $\ERT(g) = \ECT(g)$, i.e., the ERT is an extension of the ECT. Specifically, in this case, $\lfloor n\cdot(g \cdot \mathbbm{1}_{X^v_t}) \rfloor =\lceil n\cdot(g \cdot \mathbbm{1}_{X^v_t}) \rceil = n\cdot(g \cdot \mathbbm{1}_{X^v_t})$, and the integral simply becomes
\begin{align}\label{eq: the ECT is a special case of the ERT}
    \ERT(g)(v, t) \coloneqq \lim_{n \to \infty} \left\{ \frac{1}{2n} \int_X 2n\cdot(g \cdot \mathbbm{1}_{X^v_t}) \,d\chi \right\} = \int_X g \cdot \mathbbm{1}_{X^v_t} \,d\chi = \operatorname{ECT}(g)(v, t).
\end{align}

\section{Euler Characteristic of Definable Sublevel Sets}\label{section: Right continuity}

In this section, we will prove the right continuity of $t \mapsto \chi(S_t^f)$ for all definable sets $S$ and definable functions $f: S \to \bb R$ (in Section \ref{section: A General Right Continuity Result}). As a consequence, we will also prove the right continuity of the ECT (in Section \ref{section: Right Continuity of the Euler Characteristic Transform}) and subsequently the right continuity of the ERT (in Section \ref{section: Right Continuity of the Euler-Radon Transform}). We will also discuss a ``middle continuity" property of the Euler characteristic as a corollary (in Section~\ref{section: Middle Continuity of the Euler Characteristic}).

\subsection{A General Right Continuity Result}\label{section: A General Right Continuity Result}

Motivated by Theorem~\ref{thm:morse_result}, we have the following result.

\begin{thm}\label{thm:general_right_continuous}
    Let $S \subseteq \bb R^n$ be a definable set and $f: S \to \bb R$ be a definable function.
    \begin{enumerate}
        \item \label{sub:right_continuous} The function $t \mapsto \chi(S_t^f)$ is right-continuous.
        \item \label{sub:negative} There exists $C \in \bb R$ such that $\chi(S^f_t) = 0$ for all $t \leq C$.
    \end{enumerate}
\end{thm}

\begin{proof}
By the ``cell decomposition theorem" \citep[][Chapter 3, Theorem 2.11]{van1998tame}, there exists a disjoint partition of $S$ into definable subsets $S_1, ..., S_N$ such that the restriction of $f$ to each $S_i$ becomes a continuous definable function. Since the Euler characteristic is finitely additive, we have that
\[\chi(S_t^f) = \sum_{j = 1}^N \chi\left((S_j)_t^f \right).\]
For Part (\ref{sub:right_continuous}), if the function $t \mapsto \chi((S_j)_t^f)$ is right-continuous for each $j$, then the function $t \mapsto \chi(S_t^f)$ would also be right-continuous. Similarly for Part (\ref{sub:negative}), if there exists some $C_j < 0$ associated to each $S_j$ such that $\chi((S_j)_t^f) = 0$ for all $t \leq C_j$, we could take $C = \max \{C_1, ..., C_N\}$ and use the finite additivity of Euler characteristics for the general case. Both arguments reduce to the case when $f$ is continuous and definable. Thus, without loss of generality, we will hereafter assume that $f$ is a continuous definable function.

We will first prove Part (\ref{sub:right_continuous}). Since $f$ is a continuous definable function, by the ``trivialization theorem" \citep[][Chapter 9, Theorem 1.2]{van1998tame}, there is a finite partition of $\bb R$ into definable subsets $A_1, ..., A_M$; for each $A_i$, there exists some definable set $B_i$ and a definable homeomorphism $h_i: f^{-1}(A_i) \to A_i \times B_i$ such that the following diagram commutes
\begin{equation}
\begin{tikzcd}\label{definable_trivialization}
	{f^{-1}(A_i)} && {A_i \times B_i} \\
	& {A_i}
	\arrow["{h_i}", from=1-1, to=1-3]
	\arrow["\pi", from=1-3, to=2-2]
	\arrow["f", from=1-1, to=2-2]
\end{tikzcd}
\end{equation}
where $\pi:A_i \times B_i\rightarrow A_i$ denotes the standard projection. We say that $f$ is ``definably trivial" over each $A_i$ in this case.

Lemma \ref{lem:general_tame_property} indicates that $t \mapsto \chi(S^f_t)$ is piecewise constant with at most finitely many discontinuities. It suffices to show the right continuity of $t \mapsto \chi(S^f_t)$ at each discontinuity. Suppose $t \in \bb R$ is an aforementioned discontinuity. It suffices to show the following for all sufficiently small $\epsilon > 0$,
\begin{align}\label{eq: equation to prove for RC}
    \begin{aligned}
        0 &= \chi(S^{f}_{t + \epsilon}) - \chi(S^{f}_{t})\\
    &= \chi(S^{f}_{t + \epsilon} \setminus S^{f}_{t}),
    \end{aligned}
\end{align}
where the last equality follows from the finite additivity of $\chi(\cdot)$.

We observe that $f(S^{f}_{t + \epsilon} \setminus S^{f}_{t}) \subseteq (t, t + \epsilon]$ and $f^{-1}((t, t + \epsilon]) = S^{f}_{t + \epsilon} \setminus S^{f}_{t}$. Since definable subsets of $\bb R$ are precisely finite unions of points and open intervals, $(t, t + \epsilon]$ must be contained in exactly one of the $A_1, ..., A_M$ for sufficiently small $\epsilon > 0$. Without loss of generality, we assume that $(t, t+\epsilon]$ is contained in $A_1$. The diagram in Equation~\eqref{definable_trivialization} induces the following commutative diagram 
\begin{equation}
\begin{tikzcd}\label{trivialization_t_epsilon}
	{f^{-1}((t, t + \epsilon])} && {(t, t + \epsilon] \times B_1} \\
	& {(t, t + \epsilon]}
	\arrow["{\cong}", from=1-1, to=1-3]
	\arrow["\pi", from=1-3, to=2-2]
	\arrow["f"', from=1-1, to=2-2]
\end{tikzcd}
\end{equation}
where $f^{-1}((t, t + \epsilon])$ is definably homeomorphic to $(t, t+ \epsilon] \times B_1$. Since the Euler characteristic is a definable homeomorphism invariant, we have that
\begin{align*}
    \chi(S^{f}_{t + \epsilon} \setminus S^{f}_{t}) &= \chi\left(f^{-1}((t, t + \epsilon])\right)\\
    &= \chi\left((t, t + \epsilon] \times B_1 \right)\\
    &= \chi\left((t, t+\epsilon]\right) \chi(B_1) \\
    &= 0, 
\end{align*}
where the last equality follows from $\chi((t, t+\epsilon])=\chi((t, t+\epsilon))+\chi(\{t+\epsilon\})=0$.
Thus, Equation~\eqref{eq: equation to prove for RC} holds for sufficiently small $\epsilon$. We remark here that Equation (\ref{trivialization_t_epsilon}) also implies that $f(S^{f}_{t + \epsilon} \setminus S^{f}_{t})$ is empty if $B_1$ is empty and is $(t, t+\epsilon]$ if $B_1$ is non-empty.

Finally, we will prove Part (\ref{sub:negative}). Since $A_1, ..., A_M$ are finite unions of points and intervals that partition $\bb R$, there exists some $C \in \bb R$ such that for all $t \leq C$ the interval $(-\infty, t]$ is contained in exactly one of the $A_i$'s. By the same argument as above, it follows that $S^f_t$ is definably homeomorphic to $(-\infty, t] \times B_i$ for some definable set $B_i$ and for all $t < C$. Hence, by the multiplicativity of Euler characteristics,
    \[\chi(S^f_t) = \chi((-\infty, t] \times B_i)) = \chi((-\infty, t]) \times \chi(B_i) = 0. \]
The proof is completed.
\end{proof}

\begin{rem}\label{rem: known_result_analytic}
After posting the first version of this manuscript, we learned the following: the right continuity of $t\mapsto\chi(S_t^f)$ is already known when the set $S$ is an element of the o-minimal structure of globally subanalytic subsets of $\bb R^n$ and $f$ is a continuous subanalytic function. This is a consequence of Theorem 1.11 of \cite{Kashiwara_Schapira_2018} and Theorem 4.17 of \cite{Schapira_2023}; it is also a consequence of Proposition 4.18 of \cite{Schapira_2023} and Proposition 7.5 of \cite{Lebovici_2022}. Compared to the existing results, our contribution lies in the universal applicability of Theorem~\ref{thm:general_right_continuous}---our Theorem~\ref{thm:general_right_continuous}(\ref{sub:right_continuous}) applies to any sets $S$ in any o-minimal structures satisfying the axioms in Definition~\ref{def:o-minimal}. Theorem~\ref{thm:general_right_continuous} opens doors to explore o-minimal structures beyond the globally subanalytic realm. Notably, many interesting o-minimal structures do not fit into the globally subanalytic universe, including the real exponential field in \cite{Wilkies1996} and the field of real numbers with multisummable real power series in \cite{Dries_Speissegger_2000}. In fact, \cite{InfiniteOMinimal} constructed an infinite family of pairwise incompatible o-minimal structures that expands the o-minimal structure of globally subanalytic subsets. In each case, our result also extends to all definable functions in their respective o-minimal structure, which includes definable functions whose graphs are not subanalytic.
\end{rem}

One might wonder whether the discussion in the proof of Theorem \ref{thm:general_right_continuous} can be extended to any other topological invariants. As a remark, we observe that the proof of Theorem~\ref{thm:general_right_continuous} works for any real-valued homeomorphism invariant $\psi$ on definable sets provided that $\psi$ is finitely additive and $\psi((0, 1] \times B) = 0$ for any definable sets $B$. However, the following proposition shows that $\psi$ has to be the Euler characteristics multiplied by some constant.
\begin{prop}\label{prop:trivialization}
Suppose $\psi$ is a real-valued function of definable subsets of $\bb R^n$ for $n = 1, 2, ...$ such that
\begin{enumerate}
    \item $\psi$ is a homeomorphism invariant.
    \item For $A, B$ definable and disjoint, $\psi(A \cup B) = \psi(A) + \psi(B)$
    \item $\psi( (0, 1] \times (0,1)^n) = 0$ for all $n \geq 0$.
\end{enumerate}
Then $\psi$ is equal to some constant times the Euler characteristic.
\end{prop}
\begin{proof}
    For any point $x \in \mathbb{R}^n,$ the singleton $\{x\}$ is definable. By (1) in the assumption, $\psi(\{x\})$ is the same for any point $x.$ We denote $\psi(\{x\}) = \alpha.$ For any $n > 0$, we observe that \[0 = \psi((0,1] \times (0, 1)^n) = \psi((0,1)^{n+1}) + \psi(\{1\} \times (0,1)^n) = \psi(\bb R^{n+1}) + \psi(\bb R^n).\] Inductively, we have that $\psi(\bb R^n) = \alpha (-1)^n$ as \[\psi(\bb R^1) = \psi((0,1]) - \psi(\{1\}) = 0 - \alpha = -\alpha.\]
The ``cell decomposition theorem" \citep[][Chapter 3, Theorem 2.11]{van1998tame} implies that every definable set in $\mathbb{R}^n$ can be broken down into a finite disjoint union of cells. Furthermore, each cell $\mathcal{C}$ is definably homeomorphic to $\bb R^{\dim \mathcal{C}}$ \citep[][Chapter 3, ``(2.7)" therein]{van1998tame}. For any definable set $A$ with finite cell partition $C_1, ..., C_n$ given by the ``cell decomposition theorem", we have that
   \[\psi(A) = \psi\left(\bigsqcup_{i=1}^n C_i \right) = \sum_{i = 1}^n \psi(C_i) = \sum_{i=1}^n \psi(\bb R^{\dim C_i}) = \sum_{i = 1}^n \alpha (-1)^{\dim C_i} = \alpha \left(\sum_{i = 1}^n (-1)^{\dim C_i}\right) = \alpha \cdot \chi(A). \]
   Hence, we can deduce that for any definable set $A$, $\psi(A) = \alpha \cdot \chi(A)$.
\end{proof}

As a remark, we observe that the rank of cohomology with compact support satisfies assumption (1) and (3) in the statement of Proposition~\ref{prop:trivialization} but not assumption (2). For example, the rank of $H^1_c((-1, 1))$ is equal to $1$, but the ranks of $H^1_c((-1, 0) \cup (0, 1))$ and $H^1_c(\{0\})$ are $2$ and $0$ respectively.

\subsection{Right Continuity of the Euler Characteristic Transform}\label{section: Right Continuity of the Euler Characteristic Transform}

As a consequence of Theorem~\ref{thm:general_right_continuous}(\ref{sub:right_continuous}), we show the right continuity of the ECT. This is a generalization of Remark 4.14 in \cite{curry2022many} (also see Proposition 5.18 of \cite{curry2022many}, Lemma 2.3 of \cite{bestvina1997morse}, and $\S$VI.3 \cite{edelsbrunner2010computational}), which showed that the ECT is right-continuous on piecewise linearly embedded simplicial complexes. Theorem~\ref{thm:ECT_right_continuous} below only assumes that $S$ is definable, which is a much weaker assumption. In contrast, the statement of Remark 4.14 in \cite{curry2022many} assumed that $S \subseteq \bb R^n$ is setwise a compact geometric simplicial complex and thus imposed some rigidity on the geometry of $S$. For example, $\bb S^n$ is not a compact geometric simplicial complex but is definable.

\begin{thm}\label{thm:ECT_right_continuous}
    Let $S \subseteq \bb R^n$ be definable. For each fixed $v\in\mathbb{S}^{n-1}$, the following univariate function is right-continuous 
    \begin{align*}
        \ECT(S)(v, -):\ & \bb R \rightarrow \bb Z, \\
        & t \mapsto \chi(S_t^v)=\ECT(S)(v, t).
    \end{align*}
\end{thm}
\begin{proof}
Suppose $v\in\mathbb{S}^{n-1}$ is arbitrarily chosen and fixed. For this $v$, we define function $\phi_v$ by the following
\begin{align*}
    \phi_v:\ \ & S \rightarrow \mathbb{R}, \\
    & x\mapsto x\cdot v,
\end{align*}
which is continuous. The graph $\Gamma(\phi_v)$ of $\phi_v$ can be represented as follows
\begin{align*}
    \Gamma(\phi_v)= \{(x, t) \in \bb R^n \times \bb R : x \in S \text{ and } x\cdot v - t = 0\} = (S \times \bb R) \cap \{(x, t) \in \bb R^{n+1}\ |\ x\cdot v - t = 0\}.
\end{align*}
Since both $S \times \bb R$ and $\{(x, t) \in \bb R^{n+1}\ |\ v \cdot x - t = 0\}$ are definable, the graph $\Gamma(\phi_v)$ is definable, indicating that $\phi_v$ is a definable function. The right continuity of $\ECT(S)(v, -)$ then follows from Theorem~\ref{thm:general_right_continuous}(\ref{sub:right_continuous}) and choosing $f = \varphi_v$.
\end{proof}

Theorem \ref{thm:ECT_right_continuous} directly implies the following corollary, as any constructible function is a linear combination of indicator functions of definable sets \citep[e.g., see the discussion in Section 3.6 of][]{ghrist2014elementary}.
\begin{cor}
Suppose $g: X \to \bb Z$ is a constructible function. Then, $\ECT(g)(v, -): t \mapsto \operatorname{ECT}(g)(v,t)$ is right-continuous.
\end{cor}

The arguments implemented in the proofs of Theorem \ref{thm:general_right_continuous} and \ref{thm:ECT_right_continuous} do not imply the left continuity of $\ECT(S)(v, -)$. The obstacle arises from the inherent structure of the interval $(t - \epsilon, t]$, which consistently includes the endpoint $t$, irrespective of the chosen value for $\epsilon > 0$. This differs from the context of right continuity, where the right endpoint of the interval $(t, t + \epsilon]$ offers greater flexibility. Importantly, the function $\ECT(S)(v, -)$ is not left-continuous at its discontinuities, as demonstrated by the following example.
\begin{exmp}\label{exmp::ECT_not_left_cont}
    Consider the shape $S=\{x\in\mathbb{R}^2: \Vert x\Vert\le1\}\subseteq\mathbb{R}^2$. Let $v\in\mathbb{S}^1$ be arbitrarily chosen and fixed. We have the following:
    \begin{enumerate}
        \item When $t<-1$, the set $S_t^v$ is empty. Then, $\ECT(S)(v, t)=\chi(S_t^v)=0$.
        \item When $t\ge-1$, the set $S_t^v$ is non-empty and compact, and it deformation retracts to a point. Then, $\ECT(S)(v, t)=\chi(S_t^v)=1$.
    \end{enumerate}
    Therefore, for each fixed $v$, we have $\ECT(S)(v, t)=\mathbbm{1}\{t\ge-1\}$, which is a right-continuous function of $t$. However, it is not left-continuous at $t = -1.$
\end{exmp}

\subsection{Right Continuity of the Euler-Radon Transform}\label{section: Right Continuity of the Euler-Radon Transform}

Before discussing the ERT (Definition~\ref{def:ert}), we first study the LECT and SELECT (Definition~\ref{def:lect_select}) as preparation. As a result of Theorem~\ref{thm:general_right_continuous}(\ref{sub:right_continuous}) and Theorem~\ref{thm:ECT_right_continuous}, we obtain the following right continuity results.
\begin{cor}\label{cor::SE/LECT_right_continuous}
    Suppose $X \subseteq \bb R^n$ is a definable set and $g: X \to \bb R$ is a definable function. We have the following:
    \begin{enumerate}
        \item For each fixed direction $v \in \bb S^{n-1}$ and fixed $s \in \bb R$, the functions $t \mapsto \LECT(g)(v, t, s)$ and $t \mapsto \SELECT(g)(v, t, s)$ are both right-continuous.
        \item For each fixed $v\in\mathbb{S}^{n-1}$ and fixed $t \in \bb R$, the function $s \mapsto \SELECT(g)(v, t, s)$ is left-continuous.
    \end{enumerate}
\end{cor}

\begin{proof}
    According to the definition of the ECT in Equation~\eqref{eq: ECT of CF}, the function $t \mapsto \LECT(g)(v, t, s)$ can be represented using the ECT as
    \[t \mapsto \LECT(g)(v, t, s) = \ECT(\{g = s\})(v, t) = \ECT(\mathbbm{1}_{\{g = s\}})(v, t).\]
    Hence, this function is right-continuous by Theorem~\ref{thm:ECT_right_continuous}. Similarly, we may represent the function $t \mapsto \SELECT(g)(v, t, s)$ as
    \[t \mapsto \SELECT(g)(v, t, s) = \ECT(\{g \geq s\})(v, t)  = \ECT(\mathbbm{1}_{\{g \geq s\}})(v, t).\]
The proof of result (1) is completed. For $(2)$, using the notation defined in Equation~\eqref{eq: sublevel set notation}, the function $s \mapsto \SELECT(g)(v, t, s)$ can be represented as follows
    \[s \mapsto \SELECT(g)(v, t, s) = \chi\left(\{x \in X\ |\ x \cdot v \leq t, -g(x) \leq -s \} \right) = \chi\left(\{x \in X\ |\ x \cdot v \leq t \}^{-g}_{-s} \right).\]
    This function is then left-continuous by Theorem~\ref{thm:general_right_continuous}(\ref{sub:right_continuous}).
\end{proof}

Following Corollary~\ref{cor::SE/LECT_right_continuous}, we now show that the ERT proposed in \cite{meng2023Inference} is right-continuous.
\begin{thm}\label{thm:ERT_right_continuous}
    Let $X \subseteq \bb R^{n}$ be a definable set and $g: X \to \bb R$ be a bounded and compactly supported definable function. For each fixed $v \in \bb S^{n-1}$, 
    \begin{enumerate}
        \item the function $t \mapsto \ERT(g)(v, t)$ is right-continuous; hence, it is Borel measurable;
        \item for every compact interval \(\mathcal{I}=[L,U]\), the function \(t \mapsto \mathrm{ERT}(g)(v,t)\) belongs to \(L^{1}(\mathcal{I})\), i.e., it is Lebesgue integrable on \(\mathcal{I}\).
    \end{enumerate}
\end{thm}

\begin{proof}
    Part (1). Following \cite{meng2023Inference} (equivalently, using Proposition 2 of \cite{baryshnikov2010euler}), we represent the ERT using Lebesgue integrals as follows
    \begin{align}\label{eq: equation to express ERT}
        &\ERT(g)(v, t)  = \int_0^\infty G(v, t, s) \, ds, \ \ \text{ where }\\
        \notag & G(v, t, s):= \left\{\SELECT(g)(v, t, s) - \SELECT(-g)(v, t, s) \right\} +  \frac{1}{2} \left\{\LECT(-g)(v, t, s) - \LECT(g)(v, t, s)\right\}.
    \end{align}
    Since $t\mapsto G(v, t, s)$ is a finite sum of right-continuous functions by Corollary~\ref{cor::SE/LECT_right_continuous}, $G(v, t, s)$ is right-continuous with respect to $t$. 
    
    Since $g$ is a bounded function, there exists some $R > 0$ such that $G(v, t, s) = 0$ for all $|s| > R$. By Lemma~\ref{lem:SELECT_LECT_tame_property}, $G(v, t, s)$ takes only finitely many values $c_1, ..., c_n$, as $(v, t, s)$ ranges through $\bb S^{n-1} \times \bb R \times \bb R$. To apply the dominated convergence theorem (DCT), we define the dominating function $F(s)$ to be
    \[F(s) = \left(\max_{1 \leq i \leq n} |c_n|\right) \cdot \mathbbm{1}_{[-R, R]}(s).\]
    Since $\int_0^\infty F(s) ds <\infty$ and $|G(v, t, s)| \leq F(s)$ for all $(v, t, s) \in \bb S^{n-1} \times \bb R \times \bb R$, the DCT and the right continuity of $t\mapsto G(v, t, s)$ imply
    \[\lim_{\epsilon \to 0^+} \int_{0}^{\infty} \{G(v, t + \epsilon, s) - G(v, t, s)\} ds = \int_{0}^{\infty} \lim_{\epsilon \to 0^+} \{G(v, t + \epsilon, s) - G(v, t, s)\} ds = 0.\]
    Hence, we conclude that
    \begin{align*}
        \lim_{\epsilon \to 0^+} \ERT(g)(v, t + \epsilon) - \ERT(g)(v, t) &= \lim_{\epsilon \to 0^+} \int_{0}^{\infty} \{G(v, t + \epsilon, s) - G(v, t, s)\} ds = 0.
    \end{align*}
    Thus, the function $t \mapsto \ERT(g)(v, t)$ is right-continuous. We note the statement that $t \mapsto \operatorname{ERT}(g)(v, t)$ is Borel measurable was also given in \citep[][Theorem 3.3]{meng2023Inference}.
    
    Part (2). The result presented in part (2) has been previously established in the literature \citep[][Section 3]{meng2023Inference}. For the convenience of the reader and for completeness, we include a proof of part (2) here. The Borel measurability of $t \mapsto \ERT(g)(v, t)$ follows from its right continuity. In addition, since $\vert \ERT(g)(v, t) \vert = \left\vert \int_0^\infty G(v, t, s) \, ds\right\vert \le \int_0^\infty \vert G(v, t, s)\vert \, ds \le \int_0^\infty \vert F(s)\vert \, ds<\infty$, we have
    \begin{align*}
        \int_L^U  \vert \ERT(g)(v, t) \vert \,dt \le \int_L^U \left(\int_0^\infty \vert F(s)\vert \, ds\right)\,dt = (U-L)\cdot \int_0^\infty \vert F(s)\vert \, ds <\infty. 
    \end{align*}
    Therefore, the function \(t \mapsto \mathrm{ERT}(g)(v,t)\) is Lebesgue integrable.
\end{proof}

\subsection{Middle Continuity of the Euler Characteristic}\label{section: Middle Continuity of the Euler Characteristic} As an application of Theorem~\ref{thm:ECT_right_continuous}, we will prove that taking the Euler characteristic of a ``neighborhood" of a definable fiber $f^{-1}(t)$ converges to the Euler characteristic of the $f^{-1}(t)$ when ``shrinking the neighborhood". This result, stated more precisely in Proposition \ref{prop:euler_middle_cont} below, helps to connect the ECT, LECT, and SELECT (see Definitions \ref{def:ECT} and \ref{def:lect_select}).
\begin{prop}\label{prop:euler_middle_cont}
        Let $S$ be a definable subset of $\bb R^n$ and $f: S \to \bb R$ be a definable function. For any $t \in \bb R$,
        \[\lim_{\delta \to 0^+} \chi\left(f^{-1}([t-\delta, t+\delta])\right) = \chi\left(f^{-1}(t)\right).\]
\end{prop}

\begin{proof}
    Similar to the proof of Theorem~\ref{thm:general_right_continuous}, the ``cell decomposition theorem" \citep[][Chapter 3, Theorem 2.11]{van1998tame} reduces the proposition to the case where $f$ is continuous. As $f$ is continuous and definable, the ``trivialization theorem" \citep[][Chapter 9, Theorem 1.2]{van1998tame} implies that for all $\delta > 0$ sufficiently small, there exists some definable sets $A$ and $B$ such that $f^{-1}([t-\delta, t))$ and $f^{-1}((t, t+\delta]))$ are definably homeomorphic to $[t-\delta, t) \times A$ and $(t, t+\delta] \times B$, respectively. As the Euler characteristic distributes over finite Cartesian products, we have that
    \[\chi\left(f^{-1}([t-\delta, t)\right) = \chi\left([t-\delta, t)\right) \cdot \chi(A) = 0 \ \ \text{ and }\ \ \chi\left(f^{-1}((t, t+\delta])\right) = \chi\left((t, t+\delta])\right) \cdot\chi(B) = 0.\]
    By the finite additivity of the Euler characteristic, we have that
    \begin{align*}
        \chi\left(f^{-1}([t-\delta, t+\delta])\right) &= \chi\left(f^{-1}([t-\delta, t))\right) + \chi(f^{-1}(t)) + \chi\left(f^{-1}((t, t+\delta]))\right) \\
        &= 0 + \chi\left(f^{-1}(t)\right) + 0 \\
        &= \chi(f^{-1}(t)).
    \end{align*}
\end{proof}

As a corollary, we also obtain the following relationship between LECT, SELECT, and ECT.

\begin{cor}\label{cor: connect the ECT, LECT, and SELECT}
    Let $S$ be a definable subset of $\bb R^n$ and $g: S \to \bb R$ be a definable function. For any $t \in \bb R$,
    \[\lim_{\delta \to 0^+} \left\{\operatorname{SELECT}(g)(v, t, s - \delta) + \operatorname{SELECT}(-g)(v, t, -s - \delta) \right\} = \operatorname{LECT}(g)(v, t, s) + \operatorname{ECT}(S)(v,t).\]
\end{cor}

\begin{proof}
    Unwrapping the definitions, we have that
    \begin{enumerate}
        \item $\operatorname{SELECT}(g)(v, t, s - \delta) = \chi(\{x \in S^v_t\ |\ g(x) \geq s - \delta\}),$
        \item $\operatorname{SELECT}(-g)(v, t, -s-\delta) = \chi(\{x \in S^v_t\ |\ -g(x) \geq -s-\delta\}) = \chi(\{x \in S^v_t\ |\ g(x) \leq s+\delta\}),$
        \item $\operatorname{LECT}(g)(v, t, s) = \chi(\{x \in S^v_t\ |\ g(x) = s\})$.
    \end{enumerate}
    The corollary now follows from taking the limit $\delta \to 0^+$ on the following equation implied by the finite additivity of the Euler characteristic,
    \[\chi(\{x \in S^v_t\ |\ g(x) \geq s - \delta\}) + \chi(\{x \in S^v_t\ |\ g(x) \leq s + \delta\}) = \chi(S^v_t) + \chi(\{x \in S^v_t\ |\ s - \delta \leq g(x) \leq s + \delta\}). \]
\end{proof}

\section{Homotopy Type of Compact Definable Sublevel Sets}\label{section: Homotopy Types}

Proposition~\ref{prop:trivialization} shows the obstructions in generalizing the right continuity to other invariants (e.g., Betti numbers) on general definable sets and definable functions. In this section, we will restrict our attention to compact definable sets $K \subseteq\mathbb{R}^n$, which is always satisfied in practical applications of TDA (e.g., the sets $K$ represent glioblastoma multiforme tumors in \cite{crawford2020predicting}, and the sets $K$ represent mandibular molars of primates in \cite{wang2021statistical} and \cite{meng2022randomness}). The notation $K$ is preferred over $S$ to emphasize this compactness constraint.

In Section~\ref{section: Right Continuity of Homotopy Type}, we will prove in Theorem~\ref{thm:homotopy_K} that the homotopy type of definable sublevel sets on $K$ is right-continuous with respect to a continuous definable function $\Phi: K \to \bb R$. In particular, this would imply that the singular Betti numbers of definable sublevel sets on $K$ would vary right-continuously. In Section~\ref{section: Corollaries of Right Homotopy}, we will also discuss two additional corollaries of Theorem~\ref{thm:homotopy_K}.

\subsection{Right Continuity of Homotopy Type}\label{section: Right Continuity of Homotopy Type}

Motivated by Theorem~\ref{thm:morse_result}, we have the following result on the homotopy type of definable sublevel sets. 

\begin{thm}\label{thm:homotopy_K}
    Let $K$ be a compact definable subset of $\bb R^n$ and $\Phi: K \to \bb R$ be a continuous definable function.
    \begin{enumerate}
        \item \label{sub:continuous_definable}For any $t \in \bb R$, $K^{\Phi}_{t + \delta} = \Phi^{-1}((-\infty, t + \delta])$ deformation retracts to $K^{\Phi}_{t} = \Phi^{-1}((-\infty, t])$ for all $\delta > 0$ sufficiently small.
    \end{enumerate}
    In particular, for any fixed $t \in \bb R$, $v \in \bb S^{n-1}$, and $\Phi(x) = \varphi_v(x) = x \cdot v$ be as in Theorem~\ref{thm:ECT_right_continuous}, we obtain the following two consequences:
    \begin{enumerate}[resume]
        \item \label{sub:deformation_retract} The definable set $K^{v}_{t + \delta}$ deformation retracts onto $K^{v}_{t}$ for all $\delta > 0$ sufficiently small.
        \item \label{sub:betti_number}For each fixed $v \in \bb S^{n-1}$ and integer $k$, the function $t\mapsto\beta_k(K_t^v)$ is right-continuous, where $\beta_k(K_t^v)$ denotes the $k$-th Betti number of $K_t^v$. 
    \end{enumerate}
\end{thm}
\noindent Theorem~\ref{thm:homotopy_K}(\ref{sub:deformation_retract}) implies Theorem~\ref{thm:ECT_right_continuous} in the special case of compact definable sets. Theorem~\ref{thm:homotopy_K} may not be true when $K$ is not compact, which is demonstrated by the following example.
\begin{exmp}\label{exmp:counter_homotopy_non_compact}
    Consider $S = \{x \in \bb R: x > 0\} \subseteq \mathbb{R}^1$ and $\Phi(x) = v \cdot x$ where $v$ is the positive unit vector on the real line. Note that $S$ is definable but not compact.
    \begin{enumerate}
        \item For any $t > 0$, we have $S_t^v = \{x \in \mathbb{R}: 0 < x \leq t\}\ne\varnothing$.
        \item For any $t\le0$, we have $S_t^v = \varnothing$.
    \end{enumerate}
    Hence, no matter how small $\delta$ is, $S_\delta^v$ does not deformation retract onto $S_0^v$.
\end{exmp}

The proof of Theorem~\ref{thm:homotopy_K} will rely on the following lemma:
\begin{lem}[Exercise 4.11 of \cite{Coste2002ANIT}]\label{lem:local_conic_structure}
Let $Z \subseteq S$ be two closed and bounded definable sets. Let $f$ be a nonnegative continuous definable function on $S$ such that $f^{-1}(0) = Z$. Then there exists $\delta > 0$ sufficiently small and a continuous definable map $h: f^{-1}(\delta) \times [0, \delta] \to f^{-1}([0, \delta])$ such that
\begin{enumerate}
    \item \label{sub:perserve_t} $f(h(x, t)) = t$ for every $(x, t) \in f^{-1}(\delta) \times [0, \delta]$,
    \item \label{sub:perserve_end} $h(x, \delta) = x$ for every $x \in f^{-1}(\delta)$,
    \item \label{sub:homeomorphism} $h$ restricted to $f^{-1}(\delta) \times (0, \delta]$ is a homeomorphism onto $f^{-1}((0, \delta])$.
\end{enumerate}
\end{lem}
\noindent For example, the attached cylinder $f^{-1}((0, \delta])$ in Figure~\ref{fig:Fig2a} is homeomorphic to $f^{-1}(\delta) \times (0, \delta]$ in Figure~\ref{fig:Fig2b}.
\begin{figure}[h]
  \centering
  \subfloat[]{\includegraphics[width=0.4\textwidth]{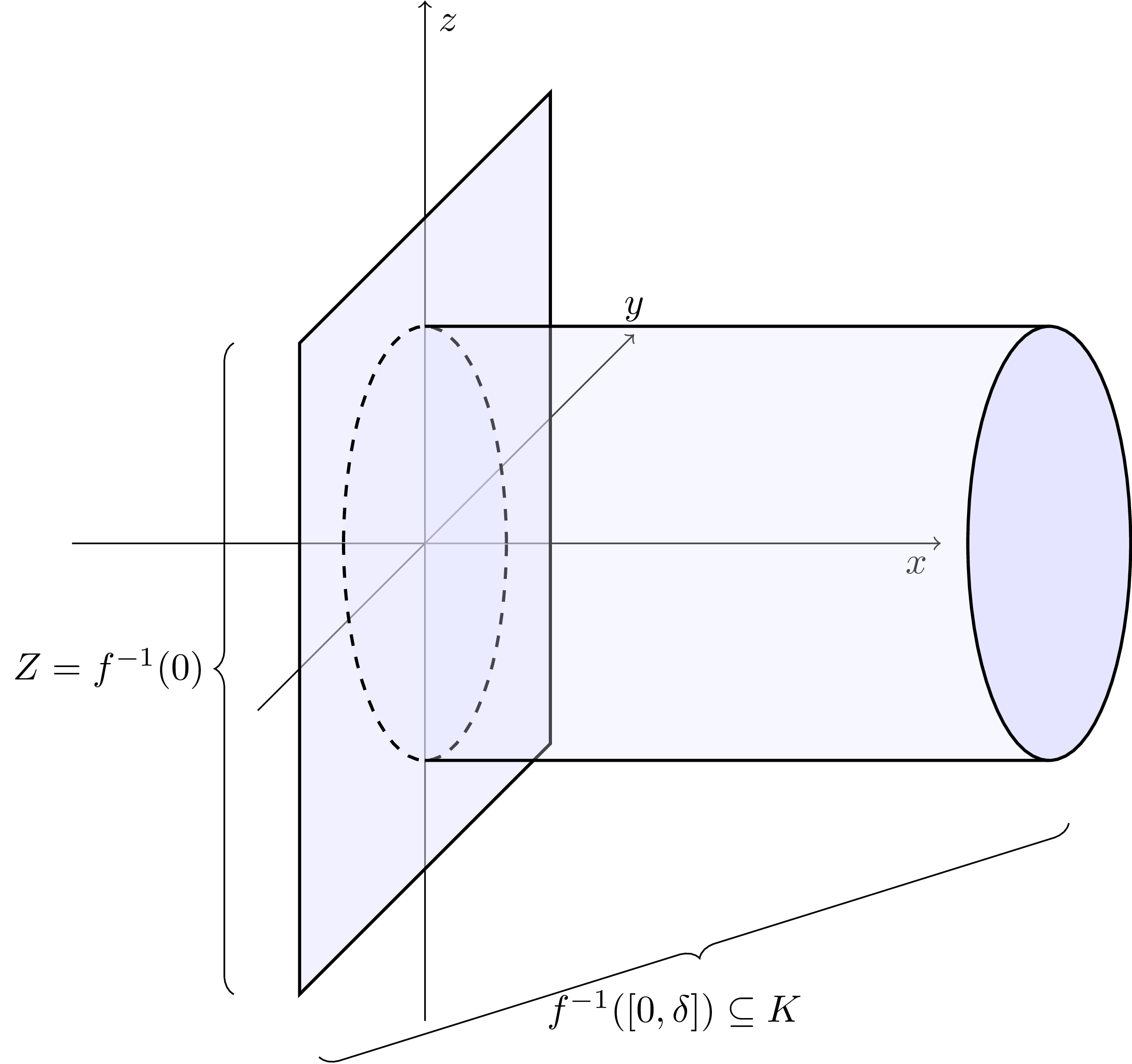}\label{fig:Fig2a}}
  \hfill
  \subfloat[]{\includegraphics[width=0.4\textwidth]{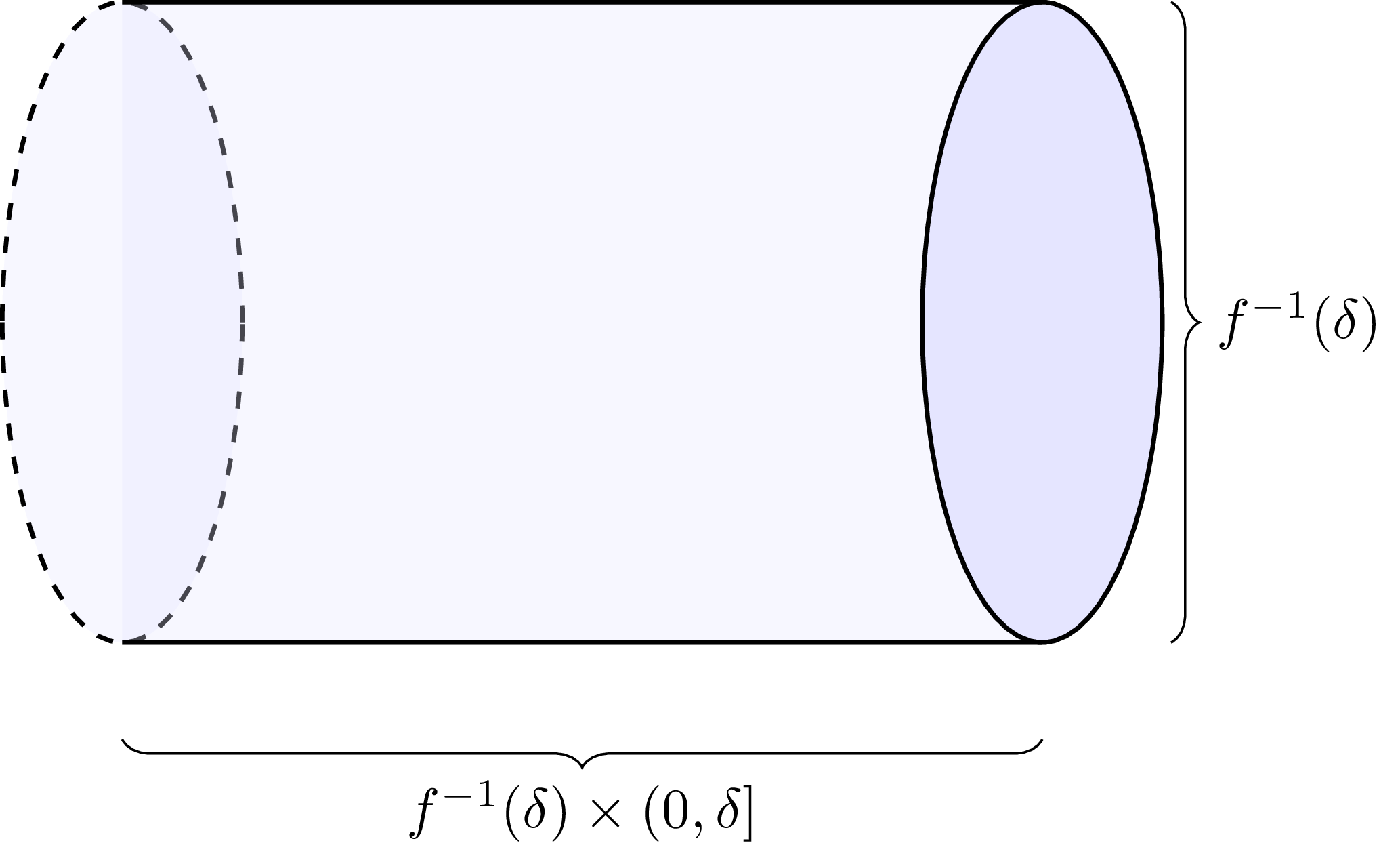}\label{fig:Fig2b}}
  \caption{An example illustrating Lemma~\ref{lem:local_conic_structure} with $f$ being the projection to $x$-coordinate. The left figure shows a solid cylinder attached to a square. The right figure shows that the solid cylinder with one base missing $f^{-1}((0,\delta]$) is definably homeomorphic to $(0, \delta] \times f^{-1}(\delta)$.}
  \label{fig:mapping_cylinder}
\end{figure}
\noindent Because this is a known result, we will leave its proof in the Appendix (Section~\ref{subsec:proof_local_conic}). We are now ready to prove Theorem~\ref{thm:homotopy_K}.
\begin{proof}[Proof of Theorem~\ref{thm:homotopy_K}]
We will first prove Part (\ref{sub:continuous_definable}). Part (\ref{sub:deformation_retract}) then follows directly from the definition of $K^v_t$, and Part (\ref{sub:betti_number}) is a direct corollary of Part (\ref{sub:deformation_retract}) as Betti numbers are homotopy invariants on compact sets.

For Theorem~\ref{thm:homotopy_K}(\ref{sub:continuous_definable}), it suffices for us to show that $\Phi^{-1}([t, t + \delta])$ deformation retracts onto $\Phi^{-1}(\{t\})$. This is because, by the ``pasting lemma" \citep[][Theorem 18.3]{Munkres_2014}, we can combine this deformation retract with the identity homotopy on $\Phi^{-1}((-\infty, t])$ to create a deformation retract of $\Phi^{-1}((-\infty, t + \delta]))$ onto $\Phi^{-1}((-\infty, t])$. 

Consider the continuous definable non-negative function $f: \Phi^{-1}([t, +\infty)) \to \bb R$ given by $f(x) = \Phi(x) - t$. We apply Lemma~\ref{lem:local_conic_structure} with $S = \Phi^{-1}([t, +\infty))$, $Z = f^{-1}(0) = \Phi^{-1}(t)$, and $f^{-1}([0, \delta]) = \Phi^{-1}([t, t + \delta])$. By Lemma~\ref{lem:local_conic_structure}, there exists $\delta > 0$ sufficiently small and a continuous definable map $h: f^{-1}(\delta) \times [0, \delta] \to f^{-1}([0, \delta])$ with properties $(1), (2), (3)$ listed in the lemma. Now consider the map
\[F: (f^{-1}(\delta) \times [0, \delta]) \sqcup f^{-1}(0) \to f^{-1}([0, \delta]) = \Phi^{-1}([t, t + \delta]), \]
whose restriction to $f^{-1}(\delta) \times [0, \delta]$ is the map $h$ and whose restriction to $f^{-1}(0)$ is the identity embedding. $F$ is a continuous map by the pasting lemma. We also observe that $F$ is surjective because $h$ restricted to $f^{-1}(\delta) \times (0, \delta]$ is a homeomorphism onto $f^{-1}((0, \delta])$ and the restriction of $F$ on $f^{-1}(0)$ is surjective onto $f^{-1}(0)$.

Since $f^{-1}(\delta) \times [0, \delta]$ and $f^{-1}(0)$ are both compact, $F$ is a continuous map from a compact topological space into a Hausdorff space. Hence, $F$ is a closed continuous surjection and is thus a quotient map. Thus, $F$ induces a homeomorphism between $f^{-1}([0, \delta])$ and the quotient space $(f^{-1}(\delta) \times [0, \delta]) \sqcup f^{-1}(0) /\sim$, where $\sim$ is given by the relation $\xi \sim \eta$ if and only if $F(\xi) = F(\eta)$. 

For ease of notation, let $P$ denote $(f^{-1}(\delta) \times [0, \delta]) \sqcup f^{-1}(0)$. We will treat the equivalence relation $\sim$ explicitly as a subset of $P \times P$. We will write $R_1 = \{(\xi, \eta) \in P \times P\ |\ F(\xi) = F(\eta)\}$ as the equivalence relation $\sim$ on $P$ given by the quotient map $F: P \to f^{-1}([0, \delta])$.

Let $g: f^{-1}(\delta) \to f^{-1}(0)$ be the continuous map defined by $g(x) = h(x, 0)$ and $M_g$ be the mapping cylinder of $g$ \citep[][Chapter 0]{hatcher2002algebraic}. Note that $g$ is well-defined as the image of $f^{-1}(\delta)$ under $g$ is contained in $f^{-1}(0)$ by Lemma~\ref{lem:local_conic_structure}(\ref{sub:perserve_t}). The mapping cylinder $M_g$ can be realized as the quotient space of $P$ with the smallest equivalence relation $\sim'$ containing the relations $(x, 0) \sim' g(x) = h(x, 0)$ for all $(x, 0) \in f^{-1}(\delta) \times \{0\}$. Explicitly, we can denote $\sim'$ as $R_2 \subseteq P \times P$, where $R_2$ is the intersection of all equivalence relations on $P$ containing the set $X = \{((x, 0), h(x, 0))\ |\ (x, 0) \in f^{-1}(\delta) \times \{0\}\}$.

We claim that $R_1 = R_2$. First of all, $F(x, 0) = h(x, 0) = id(h(x, 0)) = F(h(x, 0))$, hence $X \subseteq R_1$. Thus, $R_2$ is a subset of $R_1$. Conversely, to show that $R_1 \subseteq R_2$, it suffices for us to show that for any equivalence relation $R$ on $P$ containing $X$, $R$ contains $R_1$. Indeed, suppose $R$ contains $X$, and let $(\xi, \eta) \in P \times P$ such that $F(\xi) = F(\eta)$, then we wish to show that $(\xi, \eta) \in R$ by doing the following caseworks:
\begin{itemize}
    \item If $\xi, \eta \in (f^{-1}(\delta) \times (0, \delta]) \cup f^{-1}(0)$, the restriction of $F$ to this subspace is injective, hence $F(\xi) = F(\eta)$ implies $\xi = \eta$. Clearly $(\xi, \xi) \in R$ by reflexivity.
\end{itemize}
\quad \quad By Lemma~\ref{lem:local_conic_structure}(\ref{sub:perserve_t}), we know that the image of $f^{-1}(\delta) \times \{t\}$ under $F$ is contained in $f^{-1}(t)$. Therefore, if $\xi \in f^{-1}(\delta) \times \{0\}$ such that $F(\xi) = F(\eta)$, then $\eta$ must be in either $f^{-1}(0)$ or $f^{-1}(\delta) \times \{0\}$.
\begin{itemize}
    \item If $\xi = (x, 0) \in f^{-1}(\delta) \times \{0\}$ and $\eta \in f^{-1}(0)$, in this case $h(x, 0) = F(\xi) = F(\eta) = id(\eta) = \eta$. Hence, $(\xi, \eta) = ((x, 0), h(x, 0)) \in X \subseteq R$.
    \item If $\xi = (x, 0), \eta = (y, 0) \in f^{-1}(\delta) \times \{0\}$, in this case $h(x, 0) = F(\xi) = F(\eta) = h(y, 0)$. Since $X \subseteq R$, we know that $((x, 0), h(x, 0)), ((y, 0), h(y, 0)) \in R$. By symmetry, $(h(y, 0), (y, 0)) \in R$.

    Since $((x, 0), h(x, 0)), (h(y, 0), (y, 0)) \in R$ and $h(x, 0) = h(y, 0)$, we have that $((x, 0), (y, 0)) = (\xi, \eta) \in R$ by transitivity.
    \item Finally, if $\xi \in f^{-1}(0)$ and $\eta \in f^{-1}(\delta) \times \{0\}$, this is covered by a previous case via symmetry.
\end{itemize}
Thus, we conclude that $R_1 \subseteq R$ for any equivalence relation $R$ on $P$ that contains $X$. Therefore, $R_1 \subseteq R_2$. It then follows that $R_1 = R_2$.

Hence, it follows that $\Phi^{-1}([t, t + \delta]) = f^{-1}([0, \delta])$ is homeomorphic to $P/\sim\ = P/\sim'\ = M_g$ with a homeomorphism that is the identity on $f^{-1}(0) = \Phi^{-1}(\{t\})$. It is a well-known fact in algebraic topology \citep[][Chapter 0]{hatcher2002algebraic} that the mapping cylinder $M_g$ deformation retracts to $f^{-1}(0)$.

Thus, carrying the deformation retract in $M_g$ back to $\Phi^{-1}([t, t + \delta])$, we have that $\Phi^{-1}([t, t + \delta])$ deformation retracts to $\Phi^{-1}(\{t\})$. Hence, $K^{\Phi}_{t + \delta} = \Phi^{-1}((-\infty, t + \delta])$ deformation retracts to $K^{\Phi}_{t} = \Phi^{-1}((-\infty, t])$ for all $\delta > 0$ sufficiently small.
\end{proof}

\begin{rem}\label{rem: only_proper}
    The proof of Theorem~\ref{thm:homotopy_K} only needs that the preimages of points and closed intervals under $\Phi$ are compact. In other words, it suffices for us to remove the compactness constraint on $K$ and require the map $\Phi$ to be continuous, definable, and proper. This also aligns with the discussion of the homotopy type of sublevel sets of a smooth manifold $M$ with respect to a Morse function $f$ in \cite{milnor1963morse} (Section I.3 therein), where only the preimages of closed intervals under $f$ are assumed to be compact. In the specific o-minimal structure of globally subanalytic sets, Theorem~\ref{thm:homotopy_K}(\ref{sub:betti_number}) is also a consequence of Theorem 1.11 in \cite{Kashiwara_Schapira_2018}.
\end{rem}

\subsection{Corollaries of Theorem~\ref{thm:homotopy_K}}\label{section: Corollaries of Right Homotopy}

From Theorem~\ref{thm:homotopy_K}, we also obtain the following two corollaries. The first corollary is a ``LECT and SELECT versions" of Theorem~\ref{thm:homotopy_K}.
\begin{cor}\label{cor: LECT and SELECT versions of Theorem 4.1}
    Let $K$ be a compact definable subset of $\bb R^n$ and $g: K \to \bb R$ be a continuous, definable function.
    \begin{enumerate}
        \item The definable set $\{x \in K_{t + \delta}^v : g(x) \geq s\}$ and $\{x \in K_{t + \delta}^v : g(x) = s\}$ deformation retracts to $\{x\in K_t^v: g(x)\geq s\}$ and $\{x\in K_t^v: g(x)=s\}$, respectively, for $\delta > 0$ sufficiently small.
        \item The definable set $\{x \in K_{t}^v : g(x) \geq s-\delta\}$ deformation retracts to $\{x \in K_{t}^v : g(x) \geq s\}$ for $\delta > 0$ sufficiently small.
    \end{enumerate}
\end{cor}
\noindent The proof of Corollary \ref{cor: LECT and SELECT versions of Theorem 4.1} is a direct application of Theorem~\ref{thm:homotopy_K} and is omitted here.

The second corollary is a ``middle continuity" statement for the homotopy type of compact definable level sets. It is a generalization of Proposition~\ref{prop:euler_middle_cont} in the case of compact definable sets with continuous definable functions.
\begin{cor}\label{thm:compact_middle_cont}
    Let $K$ be a compact definable subset of $\bb R^n$ and $\Phi: K \to \bb R$ be a continuous definable function. For any $t \in \bb R$, $\Phi^{-1}([t-\delta, t+\delta])$ deformation retracts to $\Phi^{-1}(t)$ for all $\delta > 0$ sufficiently small.
\end{cor}
Corollary~\ref{thm:compact_middle_cont} is significant from an application viewpoint. Suppose a computer program is tasked to find the Betti numbers $\beta_k$ (or any other homotopy invariant) of the level sets $\Phi^{-1}(t)$ in the setup of Corollary~\ref{thm:compact_middle_cont}. Due to error and imprecision in practical applications, a computer program would typically only compute $\beta_k(\Phi^{-1}([t-\delta, t+\delta])$ with some margin of error $\delta > 0$. Corollary~\ref{thm:compact_middle_cont} thus guarantees that the computation would converge to the desired value as $\delta$ gets small, that is, 
\[\lim_{\delta \to 0^+} \beta_k\left(\Phi^{-1}([t-\delta, t+\delta])\right) = \beta_k\left(\Phi^{-1}(t)\right).\]
Similar to the discussions in Remark~\ref{rem: only_proper}, Corollary~\ref{thm:compact_middle_cont} may also be generalized to a (not necessarily compact) definable set and a continuous proper real-valued definable function on it. 

\begin{proof}[Proof of Corollary~\ref{thm:compact_middle_cont}]
    In the proof of Theorem~\ref{thm:homotopy_K}, we showed that $\Phi^{-1}([t, t+\delta])$ deformation retracts onto $\Phi^{-1}(t)$ for all $\delta > 0$ sufficiently small. It suffices for us to show that $\Phi^{-1}([t-\delta, t])$ also deformation retracts onto $\Phi^{-1}(t)$ for all $\delta > 0$ sufficiently small. Following this statement, we can use the ``pasting lemma" \citep[][Theorem 18.3]{Munkres_2014} to combine both deformation retracts to a deformation retract of $\Phi^{-1}([t-\delta, t+\delta])$ onto $\Phi^{-1}(t)$.

    Consider the continuous definable non-negative function $g: \Phi^{-1}((-\infty, t]) \to \bb R$ given by $g(x) = -\Phi(x) + t$. By the exact same arguments as in the proof of Theorem~\ref{thm:homotopy_K}, we can again apply Lemma~\ref{lem:local_conic_structure} with $S = \Phi^{-1}((-\infty, t])$, $Z = g^{-1}(0) = \Phi^{-1}(t)$, and $g^{-1}([0, \delta]) = \Phi^{-1}([t-\delta, t])$ to show that $\Phi^{-1}([t-\delta, t])$ deformation retracts onto $\Phi^{-1}(t)$ for all $\delta > 0$ sufficiently small. This concludes the proof of the corollary.
\end{proof}

\section{Applications to Topological Data Analysis}\label{section: Applications to the Smooth Euler Characteristic Transform}

In this section, we utilize the results in previous sections to study several Euler characteristic-based shape descriptors in TDA. Specifically, we investigate the following:
\begin{enumerate}
    \item The relationship between the ECT \citep{turner2014persistent, ghrist2018persistent, curry2022many} and the SECT \citep{crawford2020predicting, meng2022randomness}.
    \item The relationship between the ERT and the smooth Euler-Radon transform \citep[SERT,][]{meng2023Inference}.
\end{enumerate} 

Motivated by the SECT, \cite{meng2023Inference} introduced the SERT by smoothing the ERT via Lebesgue integration as follows: for any bounded, definable, and compactly supported function $g:\mathbb{R}^n\rightarrow\mathbb{R}$ satisfying $\operatorname{supp}(g) \subseteq B(0,R)$ and
\begin{align}\label{eq: supp(g) is strictly smaller than B}
    \operatorname{dist}\Big(\operatorname{supp}(g),\, \partial B(0,R)\Big) := \inf\Big\{\Vert x - y \Vert: x\in \operatorname{supp}(g) \text{ and }y\in \partial B(0,R)\Big\} > 0,
\end{align} 
the SERT of $g$ is defined as $\operatorname{SERT}(g):=\{\operatorname{SERT}(g)(v, t): (v,t)\in\mathbb{S}^{n-1}\times\mathbb{R}\}$, where
\begin{align}\label{eq: def of the SERT}
    \operatorname{SERT}(g)(v,t):= \int_{-R}^t \operatorname{ERT}(g)(v,\tau)\,d\tau - \frac{t+R}{2R}\int_{-R}^R \operatorname{ERT}(g)(v,\tau)\,d\tau,
\end{align}
for all $(v, t) \in \mathbb{S}^{n-1}\times [-R,\,R]$. Theorem~\ref{thm:ERT_right_continuous} ensures that the Lebesgue integrals in Equation~\eqref{eq: def of the SERT} are well defined. The SERT converts grayscale image-valued data (e.g., computerized tomography scans of tumors) into functional data. Equation~\eqref{eq: the ECT is a special case of the ERT} implies 
\begin{align}\label{eq: ECT VS ert, SECT vs SERT}
    \operatorname{ERT}(\mathbbm{1}_K) = \operatorname{ECT}(K) \ \ \ \text{ and }\ \ \ \operatorname{SERT}(\mathbbm{1}_K) = \operatorname{SECT}(K)
\end{align}
for any definable compact $K\subseteq B(0,R)$. Therefore, to investigate the relationship between the ECT and SECT, it suffices to investigate the relationship between the ERT and SERT, which is precisely described by the following theorem.
\begin{thm}\label{thm: ERT vs. SERT}
    Suppose $g:\mathbb{R}^n\rightarrow\mathbb{R}$ is a bounded, definable, and compactly supported function satisfying $\operatorname{supp}(g) \subseteq B(0,R)$ and Equation~\eqref{eq: supp(g) is strictly smaller than B}. Then, $\operatorname{SERT}(g)=\{\operatorname{SERT}(g)(v,t):\,(v,t)\in\mathbb{S}^{n-1}\times[-R,R]\}$ uniquely determines $\operatorname{ERT}(g)=\{\operatorname{ERT}(g)(v,t):\,(v,t)\in\mathbb{S}^{n-1}\times[-R,R]\}$. Hence, $\operatorname{SERT}(g)$ and $\operatorname{ERT}(g)$ uniquely determine each other.
\end{thm}
\noindent Theorem \ref{thm: ERT vs. SERT} implies that the smoothing procedure in Equation~\eqref{eq: def of the SERT} via Lebesgue integration preserves all the information within $\operatorname{ERT}(g)$. Furthermore, Theorem \ref{thm: ERT vs. SERT}, together with Equation~\eqref{eq: ECT VS ert, SECT vs SERT}, implies that the transition from the ECT to the SECT via Equation~\eqref{eq: def of SECT(S)(v,t)} is invertible. Precisely, we have the following corollary of Theorem \ref{thm: ERT vs. SERT}.
\begin{cor}\label{cor: the ECT and SECT determine each other}
    Suppose $K\subseteq\mathbb{R}^n$ is compact, definable, and bounded by the open ball $B(0,R)$. Then, $\operatorname{SECT}(K)=\{\operatorname{SECT}(K)(v,t):\,(v,t)\in\mathbb{S}^{n-1}\times[-R,R]\}$ uniquely determines $\operatorname{ECT}(K)=\{\operatorname{ECT}(K)(v,t):\,(v,t)\in\mathbb{S}^{n-1}\times[-R,R]\}$. Hence, $\operatorname{SECT}(K)$ and $\operatorname{ECT}(K)$ uniquely determine each other.
\end{cor}
\noindent Corollary \ref{cor: the ECT and SECT determine each other}, together with the success of the ECT in applications \citep[e.g.,][]{turner2014persistent, wang2021statistical}, justifies the utilization of the SECT in sciences \citep{crawford2020predicting, marsh2022detecting, meng2022randomness}. In addition, theoretical results on the ECT \citep{ghrist2018persistent, curry2022many} can be applied to the SECT via Corollary \ref{cor: the ECT and SECT determine each other}.

We employ techniques that were previously implemented by \cite{meng2023Inference} in the subsequent proof.
\begin{proof}[Proof of Theorem \ref{thm: ERT vs. SERT}]
We arbitrarily choose a direction $v\in\mathbb{S}^{n-1}$ and fix it. Equation~\eqref{eq: supp(g) is strictly smaller than B} implies that $\ERT(g)(v, t)=0$ and $\operatorname{SERT}(g)(v,t)=-\frac{t+R}{2R}\int_{-R}^R \operatorname{ERT}(g)(v,\tau)\,d\tau$ when $-R <t< -R+\operatorname{dist}\left(\operatorname{supp}(g),\, \partial B(0,R)\right)$. Equation~\eqref{eq: def of the SERT} implies that the following equation holds almost everywhere
    \begin{align}\label{eq: derivative of the SERT}
        \frac{d}{dt}\operatorname{SERT}(g)(v,t)=\ERT(g)(v, t) + \left\{-\frac{1}{2R} \int_{-R}^R \ERT(g)(v, \tau)\,d\tau \right\}.
    \end{align}
    That is, there exists a measurable subset $N$ of $\mathbb{R}$ with Lebesgue measure zero such that Equation~\eqref{eq: derivative of the SERT} holds for all $t\notin N$. Hence,
    \begin{align*}
        \lim_{t\rightarrow-R}\frac{d}{dt}\operatorname{SERT}(g)(v,t)=\left\{-\frac{1}{2R} \int_{-R}^R \ERT(g)(v, \tau)\,d\tau \right\},
    \end{align*}
    which implies
    \begin{align}\label{eq: represent the ECT via SECT}
        \ERT(g)(v, t)=\frac{d}{dt}\operatorname{SERT}(g)(v,t)-\lim_{t\rightarrow-R}\frac{d}{dt}\operatorname{SERT}(g)(v,t),\ \ \text{ for all }t\notin N.
    \end{align}
    The right continuity of $t\mapsto \operatorname{ERT}(g)(v,t)$ implies that Equation~\eqref{eq: represent the ECT via SECT} recovers the values of $\ERT(g)(v, t)$ for $t\in N$. Specifically, for each $t\in N$, there exists a sequence $\{t_n\}_{n\ge1}$ such that
    \begin{itemize}
        \item $t_n>t$ for all $n\ge1$,
        \item $t_n \notin N$ for all $n\ge 1$, and
        \item $t_n\rightarrow t$ as $n\rightarrow\infty$,
    \end{itemize}
    since $N$ has Lebesgue measure zero. Then, by the right continuity of $t\mapsto \operatorname{ERT}(g)(v,t)$, we have
     \begin{align}\label{eq: the values in N through the right continuity}
         \operatorname{ERT}(g)(v,t) = \lim_{n\rightarrow\infty} \operatorname{ERT}(g)(v,t_n), \quad \text{where }t\in N,
     \end{align}
     where each $\operatorname{ERT}(g)(v,t_n)$ is determined by Equation~\eqref{eq: represent the ECT via SECT}. Therefore, the $\operatorname{ERT}(g)$ is represented by $\operatorname{SERT}(g)$ through Equation~\eqref{eq: represent the ECT via SECT} and \eqref{eq: the values in N through the right continuity}.
\end{proof}

\section{Discussion}\label{section: Discussion}

In this article, we studied definable sublevel sets from both pure perspectives---Euler characteristics and homotopy types---and applied viewpoints---topological descriptors developed in the TDA literature. Our results contribute to the future probabilistic development of TDA. For example, the function $t \mapsto \chi(S^v_t)$ becomes a stochastic process for each fixed $v\in\mathbb{S}^{n-1}$ if the shape-valued data $S$ is viewed as random. Then, our Theorem \ref{thm:ECT_right_continuous}, combined with Lemma \ref{lem:general_tame_property}, guarantees that the sample paths of this process are right-continuous with left limits (càdlàg). The theory of stochastic processes with càdlàg sample paths has been extensively studied \citep[][Section 21.4]{klenke2013probability}, paving the way for the examination of the probabilistic properties of the topological descriptor ECT.

\subsection*{Acknowledgements} M.J. would like to thank Dr.~Mark Ainsworth and Dr.~Paul Dupuis for helpful discussions and comments, Dr.~Thomas Goodwillie for helpful conversations on piecewise linear topology, and both Dr.~Nicole Looper and Dr.~Richard Schwartz for helpful conversations and advice on the contexts of this paper. K.M. would like to thank Dr.~Ningchuan Zhang for helpful conversations on Morse theory. We thank Dr.~Pierre Schapira for helpful conversations on microlocal sheaf theory and constructible functions. We also thank Dr.~Vadim Lebovici for helpful conversations on Euler calculus and real analytic geometry, especially for providing references to Remark~\ref{rem: known_result_analytic} and Remark~\ref{rem: only_proper}. We also thank Kexin Ding for helpful questions and comments on Section~\ref{section: Right continuity} and~\ref{section: Homotopy Types}, for the partial writings of what is now Proposition~\ref{prop:trivialization}, Example~\ref{exmp::ECT_not_left_cont}, and Example~\ref{exmp:counter_homotopy_non_compact}, for partially organizing the references, and for participating in the making of Figure~\ref{fig:sublevel_set} and Figure~\ref{fig:mapping_cylinder}.

\section*{Declarations}

\subsection*{Ethical Approval}

Not applicable.

\subsection*{Competing Interests}

The authors declare no competing interests.

\subsection*{Author Contributions}

K.M. conceived and supervised the study. K.M. asked the original question, now formulated as Theorem~\ref{thm:ECT_right_continuous}, that M.J. proved and generalized upon. K.M. and M.J. wrote the introduction. M.J. made the primary contributions to the writing of Section~\ref{section: Background: O-minimal Structures and Euler Calculus}, the proofs of the results in Section~\ref{section: Right continuity}, and the proofs of the results in Section~\ref{section: Homotopy Types}. K.M. made the primary contributions to the proofs in Section~\ref{section: Applications to the Smooth Euler Characteristic Transform} and wrote Section~\ref{section: Discussion}. M.J. wrote the Appendix. All authors reviewed and revised the writings in the manuscript.

\subsection*{Funding}

None.

\subsection*{Availability of Data and Materials}

Not applicable.
\appendix
\section{Additional Content}

\subsection{Proof of Lemma~\ref{lem:local_conic_structure}}\label{subsec:proof_local_conic} Our proof will also use concepts in triangulations and simplicial complexes. We refer the reader to \cite{Rourke_Sanderson_1982} for a thorough treatment of the subject. For our purposes, we denote the geometric (closed) simplex $[a_0, ..., a_d]$ as the convex hull of affinely independent points $a_0, ..., a_d$, i.e.,
    \[[a_0, ..., a_d] = \left\{x \in \bb R^n: x = \lambda_0 a_0 + ... + \lambda_d a_d,\, \sum_{i = 0}^d \lambda_i = 1, \lambda_i \in [0, 1] \right\}. \]

We will now prove Lemma~\ref{lem:local_conic_structure} as follows.
\begin{proof}[Proof]
We follow the hints outlined in the exercise. By the triangulation theorem for definable functions \citep[][Theorem 2]{Coste1998}, there exists a finite simplicial complex $L_S$ with a definable homeomorphism $\rho: |L_S| \to S$ from the polyhedron of $L_S$ to $S$ such that $f \circ \rho$ is linear on each simplex of $L_S$. Furthermore, this triangulation may be chosen so that $Z$ is definably homeomorphic to $|L_Z|$, where $L_Z$ is a full subcomplex of $L_S$. 

We outline the proof in several steps:
\begin{enumerate}
    \item For $(x, t) \in f^{-1}(\delta) \times [0, \delta]$, we will construct $h(x, t)$ and show it is well-defined in this section. Since $L_S$ is a finite triangulation, we may choose $\delta > 0$ sufficiently small that any vertex $v$ of $L_S$ that is not in $L_Z$ satisfies $f(\rho(v)) > \delta$.
\begin{enumerate}
    \item Let $x' = \rho^{-1}(x)$ denote the element $x$ on $|L_S|$.
    \item  By the construction of $\delta$, $x'$ is contained in a simplex $[a_0, ..., a_d]$ such that $x' = \sum_{i = 0}^d \lambda_i a_i$ in barycentric coordinates. We may assume without loss that $a_i \in L_Z$ for $i = 0, ..., k$ and $a_i \notin L_Z$ for $i = k + 1, ..., d$. 
    \item Let $\alpha = \sum_{i = 0}^k \lambda_i$. We claim that $0 < \alpha < 1$. Indeed, if $\alpha = 0$, since each $\lambda_i \geq 0$, this means that $x'$ is contained in the simplex $[a_{k+1}, ..., a_d]$. But this would mean that $f(\rho(x')) = f(x) > \delta$. On the other hand, if $\alpha = 1$, then this would mean similarly that $x'$ is contained in the simplex $[a_{1}, ..., a_k]$, so $f(\rho(x')) = f(x) = 0$. Thus, we conclude that $0 < \alpha < 1$.
    \item Now consider the point $q(x') = \sum_{i = 0}^k \left(\frac{\lambda_i}{\alpha}\right) a_i$ in the simplex $[a_0, ..., a_k] \subseteq [a_0, ..., a_d]$. Since a (geometric) simplex is convex, we may define a line from $x'$ to $q(x')$. We then define $h(x, t) = \rho(y')$, where $y'$ is the point on the line such that $f(\rho(y')) = t$.
    \item We note that such $y'$ must exist and is uniquely determined by $x'$ and $q(x')$. Indeed, this is because $f \circ \rho$ is a linear function on the simplex $[a_0, ..., a_d]$. Hence,
    \[f \circ \rho(q(x')) = \sum_{i = 0}^k \frac{\lambda_i}{\alpha} f \circ \rho(a_i) = \sum_{i = 0}^k 0 = 0. \]
    We also know that $f \circ \rho(x') = f(x) = \delta$. Since our function is continuous, such $y'$ must exist. Furthermore, this $y'$ is uniquely determined by $x'$ and $q(x')$ since our function is linear.
    \item We also need to check that $h(x, t)$ is independent of which simplex $x'$ is contained in. Indeed, suppose $x'$ is contained in the simplex $[a_0, ..., a_d]$ and $[b_0, ..., b_e]$, then $x' \in [a_0, ..., a_d] \cap [b_0, ..., b_e]$, which is also a simplex spanned by some common vertices in $\{a_0, ..., a_d\} \cap \{b_0, ..., b_e\}$. Let's say $x' \in \{c_0, ..., c_f\}$.
    
    Let $q_a(x'), q_b(x'), q_c(x')$ be the correspondent point in $Z$ given by each of the three vertex sets. We see that $q_a(x') = q_c(x') = q_b(x')$. This is because, if we were to write $x' = \sum_{i = 0}^d a_i \lambda_i = \sum_{j = 0}^e b_j \xi_j$ in terms of the barycentric coordinates of $[a_0, ..., a_d]$ and $[b_0, .., b_e]$, the vertices where $a_i$ is non-zero and the vertices where $b_j$ is non-zero are all in $\{c_0, ..., c_f\}$.
\end{enumerate}
    \item Now that we have constructed $h(x, t)$. We will verify Lemma~\ref{lem:local_conic_structure}(\ref{sub:perserve_t}) and Lemma~\ref{lem:local_conic_structure}(\ref{sub:perserve_end}):
\begin{itemize}
    \item For Lemma~\ref{lem:local_conic_structure}(\ref{sub:perserve_t}), clearly by definition $f(h(x, t)) = f(\rho(y')) = t$.
    \item For Lemma~\ref{lem:local_conic_structure}(\ref{sub:perserve_end}), clearly $f(x) = f(\rho(x')) = \delta$ and $x'$ is on the line between $x'$ and $q(x')$. Thus, $h(x, \delta) = x$.
\end{itemize}

\item Now we will check that $h(x,t)$ is continuous definable. We first define a natural map $Q: f^{-1}(\delta) \to f^{-1}([0, \delta])$ such that $Q(x) = \rho(q(\rho^{-1}(x))) = \rho(q(x'))$ , where $q(x')$ is the same well-defined element we specified in the construction of $h(x, t)$.

To show that $Q$ is continuous and definable, it suffices for us to show that $q: \rho^{-1}(f^{-1}(\delta)) \to \rho^{-1}(f^{-1}([0, \delta]))$ is continuous and definable.

We can show that $q$ is a continuous definable function by considering the Pasting Lemma \citep[][Theorem 18.3]{Munkres_2014} on the restriction of $q$ on each of $[a_0, ..., a_d] \cap \rho^{-1}(f^{-1}(\delta))$. Write $x' = \sum_{i = 1}^d \lambda_i a_i$. In terms of barycentric coordinate, the map $x' \mapsto q(x')$ is just $(\lambda_1, ..., \lambda_d) \mapsto (\frac{\lambda_1}{\alpha}, ..., \frac{\lambda_k}{\alpha}, 0, ..., 0)$, which is clearly continuous and definable.

Let $i: f^{-1}(\delta) \to f^{-1}([0, \delta])$ be the standard inclusion map. We note that $h(x, t): f^{-1}(\delta) \times [0, \delta] \to f^{-1}([0, \delta])$ is the ``straight line" homotopy between the functions $i$ and $Q$ (Technically, $\rho^{-1} \circ h$ is a straight line homotopy between $\rho^{-1} \circ i$ and $\rho^{-1} \circ Q$ on $|L_Z|$ under the definable homeomorphism $\rho$). Thus, $h(x, t)$ is also a continuous definable function.

\item Finally, we will prove Lemma~\ref{lem:local_conic_structure}(\ref{sub:homeomorphism}). We first want to prove that $h$ is bijective on $f^{-1}(\delta) \times (0, \delta] \to f^{-1}((0, \delta])$.

Indeed, let $y \in f^{-1}((0, \delta])$, since $0 < f(y) = t_0 \leq \delta$, we can find a simplex $[a_0, ..., a_d]$ containing $y' \coloneqq \rho^{-1}(y)$.

We will write $y' = \sum_{i = 1}^d r_i a_i$ in terms of barycentric coordinates. We wish to find an element $x' \in [a_0, ..., a_d]$ such that $y'$ lies between $x'$ and $q(x')$. This will prove surjectivity.

Indeed, consider an arbitrary element $x' = \sum_{i = 1}^d \lambda_i a_i$. The condition that $y'$ lies between $x'$ and $q(x')$ is true if and only if, in terms of barycentric coordinates,
\begin{equation}\label{eq:h_inverse}
  (\lambda_1, ..., \lambda_d) + \frac{\delta - t_0}{\delta} \left(\frac{\lambda_1}{\alpha} - \lambda_1, ..., \frac{\lambda_k}{\alpha} - \lambda_k, -\lambda_{k+1}, ..., -\lambda_d \right) = (r_1, ..., r_d).
\end{equation}

We then see that this uniquely determines the coefficients $\lambda_1, ..., \lambda_d$. The fact that $f(y) = f(\rho(y')) > 0$ is crucial here. If $f(\rho(y')) = 0$, then the coefficients $\lambda_{k+1}, ..., \lambda_d$ cannot be determined.

This also proves injectivity since if there's any $z', q(z')$ in $[b_0, ..., b_e]$ such that $y'$ is also in the line segment between $z'$ and $q(z')$. Then $y' \in [a_0, ..., a_d] \cap [b_0, ..., b_e] = [c_0, ..., c_f]$. Suppose for contradiction that $z'$ is not in $[c_0, ..., c_f]$, then the line from $z'$ to $q(z')$ would not be contained in $[b_0, ..., b_e]$. Thus $z'$ must be an element of $[a_0, ..., a_d]$, which forces it to have the same coefficients $\lambda_1, ..., \lambda_d$. Thus, we conclude that $h$ is bijective on $f^{-1}(\delta) \times (0, \delta]$.

\item Now consider the inverse of $h$ on $f^{-1}((0, \delta])$. We can show $h^{-1}$ is continuous by showing $h^{-1} \circ \rho$ is continuous on $\rho^{-1}(f^{-1}((0, \delta]))$. For ease of notation, we will call this space $X = \rho^{-1}(f^{-1}((0, \delta]))$.

We can verify this using the Pasting Lemma \citep[][Theorem 18.3]{Munkres_2014} on the intersection of $X$ with each simplex $[a_0, ..., a_d]$. Now $h^{-1} \circ \rho$ restricted to $[a_0, ..., a_d] \cap X$ is a function from $[a_0, ..., a_d] \cap X \to f^{-1}(\delta) \times (0, \delta]$.

By the definition of product topology, it suffices for us to verify that the map is coordinate-wise continuous. The composition of $h^{-1} \circ \rho$ and projection to $(0, \delta]$ is clearly continuous because this is the same as the map $f \circ \rho$. Indeed, for any $y' \in [a_0, ..., a_d] \cap X$ such that $\rho(y') = h(x, t)$,
\[f \circ \rho(y') = f(\rho(y')) = f(h(x, t)) = t, \quad \text{by Lemma~\ref{lem:local_conic_structure}(\ref{sub:perserve_t})}\]

As for the first coordinate $f^{-1}(\delta)$. For any $y' = \sum_{i = 1}^d r_i a_i \in [a_0, ..., a_d] \cap X$ such that $\rho(y') = h(x, t)$, Equation~(\ref{eq:h_inverse}) shows that the map from $y'$ to $\rho^{-1}(x) = \sum_{i = 1}^d \lambda_i a_i$ is continuous definable.
\end{enumerate}
\end{proof}

\bibliographystyle{abbrv}
\bibliography{references}

@article{gorecki2015comparison,
  title={A comparison of tests for the one-way ANOVA problem for functional data},
  author={G{\'o}recki, Tomasz and Smaga, {\L}ukasz},
  journal={Computational Statistics},
  volume={30},
  pages={987--1010},
  year={2015},
  publisher={Springer}
}

@article{grenander1998computational,
  title={Computational anatomy: An emerging discipline},
  author={Grenander, Ulf and Miller, Michael I},
  journal={Quarterly of applied mathematics},
  volume={56},
  number={4},
  pages={617--694},
  year={1998}
}

@article{munch2023invitation,
  title={An invitation to the Euler characteristic transform},
  author={Munch, Elizabeth},
  journal={The American Mathematical Monthly},
  volume={132},
  number={1},
  pages={15--25},
  year={2025},
  publisher={Taylor \& Francis}
}

@article{bestvina1997morse,
  title={{Morse theory and finiteness properties of groups}},
  author={Bestvina, Mladen and Brady, Noel},
  journal={Inventiones mathematicae},
  volume={129},
  pages={445--470},
  year={1997},
  publisher={Springer}
}

@book{hsing2015theoretical,
  title={{Theoretical Foundations of Functional Data Analysis, with an Introduction to Linear Operators}},
  author={Hsing, Tailen and Eubank, Randall},
  volume={997},
  year={2015},
  publisher={John Wiley \& Sons}
}

@inproceedings{jiang2020weighted,
  title={{The Weighted {E}uler Curve Transform for Shape and Image Analysis}},
  author={Jiang, Qitong and Kurtek, Sebastian and Needham, Tom},
  booktitle={Proceedings of the IEEE/CVF Conference on Computer Vision and Pattern Recognition Workshops},
  pages={844--845},
  year={2020}
}

@article{dries_real, title={{O-minimal Structures and Real Analytic Geometry}}, volume={1998}, 
@DOI={10.4310/cdm.1998.v1998.n1.a4}, number={1}, journal={Current Developments in Mathematics}, author={van den Dries, Lou}, year={1998}, pages={105–152}}

@book{TarskiMcKinsey+1951,
@url = {https://doi.org/10.1525/9780520348097},
title = {{A Decision Method for Elementary Algebra and Geometry}},
author = {Alfred Tarski and J. C. C. McKinsey},
publisher = {University of California Press},
address = {Berkeley},
@doi = {doi:10.1525/9780520348097},
isbn = {9780520348097},
year = {1951}
}

@book{Coste2002ANIT,
  TITLE = {{An Introduction to O-minimal Geometry}},
  AUTHOR = {Coste, Michel},
  URL = {https://univ-rennes.hal.science/hal-05413940},
  YEAR = {1999},
  MONTH = Nov,
  HAL_ID = {hal-05413940},
  HAL_VERSION = {v1},
  publisher={HAL open science}
}

@article{Coste1998,
author = {Coste, Michel},
journal = {Banach Center Publications},
keywords = {fewnomials; finiteness results; polynomials of bounded degree; bounded number of monomials},
language = {eng},
number = {1},
pages = {81-92},
title = {{Topological Types of Fewnomials}},
@url = {http://eudml.org/doc/208895},
volume = {44},
year = {1998}
}

@article{curry2012euler,
      title={{Euler Calculus with Applications to Signals and Sensing}}, 
      author={Justin Curry and Robert Ghrist and Michael Robinson},
      year={2012},
      journal={Proceedings of Symposia in Applied Mathematics},
      volume={70}
}

@book{milnor1963morse,
  title={{Morse theory}},
  author={Milnor, John Willard},
  number={51},
  year={1963},
  series={Annals of Mathematics Studies},
  publisher={Princeton university press}
}

@book{bredon2012sheaf,
  title={{Sheaf Theory}},
  author={Bredon, Glen E},
  volume={170},
  year={2012},
  publisher={Springer Science \& Business Media}
}

@article{baryshnikov2010euler,
author = {Yuliy Baryshnikov  and Robert Ghrist },
title = {Euler integration over definable functions},
journal = {Proceedings of the National Academy of Sciences},
volume = {107},
number = {21},
pages = {9525-9530},
year = {2010},
eprint = {https://www.pnas.org/doi/pdf/10.1073/pnas.0910927107}}

@book{ghrist2014elementary,
  title={{Elementary Applied Topology}},
  author={Ghrist, Robert W},
  volume={1},
  year={2014},
  publisher={Createspace Seattle}
}

@inproceedings{schapira1995tomography,
  title={{Tomography of Constructible Functions}},
  author={Schapira, Pierre},
  booktitle={International Symposium on Applied Algebra, Algebraic Algorithms, and Error-Correcting Codes},
  pages={427--435},
  year={1995},
  organization={Springer}
}

@book{van1998tame,
  title={{Tame Topology and O-minimal Structures}},
  author={van den Dries, Lou},
  volume={248},
  year={1998},
  publisher={Cambridge university press}
}

@article{marsh2022detecting,
  title={Detecting temporal shape changes with the Euler characteristic transform},
  author={Marsh, Lewis and Zhou, Felix Y and Qin, Xiao and Lu, Xin and Byrne, Helen M and Harrington, Heather A},
  journal={Transactions of Mathematics and its Applications},
  volume={8},
  number={2},
  pages={tnae002},
  year={2024},
  publisher={Oxford University Press}
}

@article{kirveslahti2023representing,
  title={{Representing Fields without Correspondences: the Lifted {E}uler Characteristic Transform}},
  author={Kirveslahti, Henry and Mukherjee, Sayan},
  journal={Journal of Applied and Computational Topology},
  pages={1--34},
  year={2023},
  publisher={Springer}
}

@article{meng2023Inference,
  title={{Statistical Inference on Grayscale Images via the {E}uler-{R}adon Transform}},
  author={Meng, Kun and Ji, Mattie and Wang, Jinyu and Ding, Kexin and Kirveslahti, Henry and Eloyan, Ani and Crawford, Lorin},
  journal={arXiv preprint arXiv:2308.14249v1},
  year={2023}
}

@book{williams2006gaussian,
    author = {Rasmussen, Carl Edward and Williams, Christopher K. I.},
    title = {Gaussian Processes for Machine Learning},
    publisher = {The MIT Press},
    year = {2005},
    month = {11},
    isbn = {9780262256834},
    doi = {10.7551/mitpress/3206.001.0001},
    url = {https://doi.org/10.7551/mitpress/3206.001.0001},
    eprint = {https://direct.mit.edu/book-pdf/2514321/book_9780262256834.pdf},
}

@article{dupuis1998variational,
  title={{Variational Problems on Flows of Diffeomorphisms for Image Matching}},
  author={Dupuis, Paul and Grenander, Ulf and Miller, Michael I},
  journal={Quarterly of applied mathematics},
  pages={587--600},
  year={1998},
  publisher={JSTOR}
}

@article{kendall1989survey,
  title={{A Survey of the Statistical Theory of Shape}},
  author={Kendall, David G},
  journal={Statistical Science},
  volume={4},
  number={2},
  pages={87--99},
  year={1989},
  publisher={Institute of Mathematical Statistics}
}

@article{gao2019gaussian,
  title={{Gaussian Process Landmarking on Manifolds}},
  author={Gao, Tingran and Kovalsky, Shahar Z and Daubechies, Ingrid},
  journal={SIAM Journal on Mathematics of Data Science},
  volume={1},
  number={1},
  pages={208--236},
  year={2019},
  publisher={SIAM}
}

@article{gao2019gaussianmorphometrics,
  title={{Gaussian Process Landmarking for Three-Dimensional Geometric Morphometrics}},
  author={Gao, Tingran and Kovalsky, Shahar Z and Boyer, Doug M and Daubechies, Ingrid},
  journal={SIAM Journal on Mathematics of Data Science},
  volume={1},
  number={1},
  pages={237--267},
  year={2019},
  publisher={SIAM}
}

@article{curry2022many,
  title={{How Many Directions Determine a Shape and Other Sufficiency Results for Two Topological Transforms}},
  author={Curry, Justin and Mukherjee, Sayan and Turner, Katharine},
  journal={Transactions of the American Mathematical Society, Series B},
  volume={9},
  number={32},
  pages={1006--1043},
  year={2022}
}

@article{wang2021statistical,
  title={{A Statistical Pipeline for Identifying Physical Features That Differentiate Classes of 3d Shapes}},
  author={Wang, Bruce and Sudijono, Timothy and Kirveslahti, Henry and Gao, Tingran and Boyer, Douglas M and Mukherjee, Sayan and Crawford, Lorin},
  journal={The Annals of Applied Statistics},
  volume={15},
  number={2},
  pages={638--661},
  year={2021},
  publisher={Institute of Mathematical Statistics}
}

@book{edelsbrunner2010computational,
  title={{Computational Topology: An Introduction}},
  author={Edelsbrunner, Herbert and Harer, John},
  year={2010},
  publisher={American Mathematical Soc.}
}

@book{klenke2013probability,
  title={{Probability Theory: A Comprehensive Course}},
  author={Klenke, Achim},
  year={2020},
  publisher={Springer}
}

@book{hatcher2002algebraic,
  title={{Algebraic Topology}},
  author={Hatcher, Allen},
  year={2002},
  publisher={New York : Cambridge University Press}
}

@article{hairer2009introduction,
  title={{An Introduction to Stochastic PDEs}},
  author={Hairer, Martin},
  journal={arXiv preprint arXiv:0907.4178},
  year={2009}
}

@article{ghrist2018persistent,
  title={{Persistent Homology and {E}uler Integral Transforms}},
  author={Ghrist, Robert and Levanger, Rachel and Mai, Huy},
  journal={Journal of Applied and Computational Topology},
  volume={2},
  pages={55--60},
  year={2018},
  publisher={Springer}
}

@article{turner2014persistent,
  title={{Persistent Homology Transform for Modeling Shapes and Surfaces}},
  author={Turner, Katharine and Mukherjee, Sayan and Boyer, Doug M},
  journal={Information and Inference: A Journal of the IMA},
  volume={3},
  number={4},
  pages={310--344},
  year={2014},
  publisher={Oxford University Press}
}

@article{crawford2020predicting,
  title={{Predicting Clinical Outcomes in Glioblastoma: An Application of Topological and Functional Data Analysis}},
  author={Crawford, Lorin and Monod, Anthea and Chen, Andrew X and Mukherjee, Sayan and Rabad{\'a}n, Ra{\'u}l},
  journal={Journal of the American Statistical Association},
  volume={115},
  number={531},
  pages={1139--1150},
  year={2020},
  publisher={Taylor \& Francis}
}

@article{meng2022randomness,
author = {Meng, Kun and Wang, Jinyu and Crawford, Lorin and Eloyan, Ani},
title = {Randomness of Shapes and Statistical Inference on Shapes via the Smooth Euler Characteristic Transform},
journal = {Journal of the American Statistical Association},
volume = {120},
number = {549},
pages = {498--510},
year = {2025},
publisher = {ASA Website},
doi = {10.1080/01621459.2024.2353947},


URL = { 
    
        https://doi.org/10.1080/01621459.2024.2353947
    
    

},
eprint = { 
    
        https://doi.org/10.1080/01621459.2024.2353947
    
    

}

}

@book{Rourke_Sanderson_1982, place={Berlín}, title={{Introduction to Piecewise-linear Topology}}, publisher={Springer}, author={Rourke, C. P. and Sanderson, B. J.}, year={1982}}

@book{Munkres_2014, edition={2nd}, title={{Topology}}, publisher={Pearson New International Edition}, author={Munkres, James}, year={2014}}

@article{Morse1929,
 ISSN = {00029947},
 @URL = {http://www.jstor.org/stable/1989523},
 author = {Marston Morse},
 journal = {Transactions of the American Mathematical Society},
 number = {3},
 pages = {379--404},
 publisher = {American Mathematical Society},
 title = {{The Foundations of the Calculus of Variations in the Large in m-Space (First Paper)}},
 urldate = {2023-09-04},
 volume = {31},
 year = {1929}
}

@book{matsumoto2002introduction,
  title={{An Introduction to Morse Theory}},
  author={Matsumoto, Y.},
  isbn={9781470446338},
  series={Iwanami series in modern mathematics},
  @url={https://books.google.com/books?id=RDIEuAEACAAJ},
  year={2002},
  publisher={American Mathematical Society}
}

@article{Forman2002,
author = {Forman, Robin},
journal = {Séminaire Lotharingien de Combinatoire [electronic only]},
language = {eng},
pages = {B48c, 35 p., electronic only-B48c, 35 p., electronic only},
publisher = {Universität Wien, Fakultät für Mathematik},
title = {{A User's Guide to Discrete Morse theory}},
@url = {http://eudml.org/doc/123837},
volume = {48},
year = {2002},
}

@article{Cox2003TopologicalZO,
  title={{Topological Zone Organization of Scalar Volume Data}},
  author={James Cox and Daniel B. Karron and Nazma Ferdous},
  journal={Journal of Mathematical Imaging and Vision},
  year={2003},
  volume={18},
  pages={95-117},
  @url={https://api.semanticscholar.org/CorpusID:24983543}
}

@article{Viro_1988, title={{Some Integral Calculus Based on Euler Characteristic}}, 
@DOI={10.1007/bfb0082775}, 
journal={Lecture Notes in Mathematics}, author={Viro, O. Y.}, year={1988}, pages={127–138}}

@article{Schapira_1991, title={{Operations on Constructible Functions}}, volume={72}, 
@DOI={10.1016/0022-4049(91)90131-k}, 
number={1}, journal={Journal of Pure and Applied Algebra}, author={Schapira, Pierre}, year={1991}, pages={83–93}}

@book{Kashiwara_Schapira_1990, place={Berlin}, title={{Sheaves on Manifolds}}, publisher={Springer}, author={Kashiwara, Masaki and Schapira, Pierre}, year={1990}}

@article{Kashiwara_Schapira_2018, title={{Persistent Homology and Microlocal Sheaf Theory}}, volume={2}, 
@DOI={10.1007/s41468-018-0019-z}, 
number={1–2}, journal={Journal of Applied and Computational Topology}, author={Kashiwara, Masaki and Schapira, Pierre}, year={2018}, pages={83–113}}

@article{Schapira_2023, title={{Constructible Sheaves and Functions Up to Infinity}}, 
@DOI={10.1007/s41468-023-00121-0}, 
journal={Journal of Applied and Computational Topology}, author={Schapira, Pierre}, year={2023}}

@article{Lebovici_2022, title={{Hybrid Transforms of Constructible Functions}}, 
@DOI={10.1007/s10208-022-09596-2}, 
journal={Foundations of Computational Mathematics}, author={Lebovici, Vadim}, year={2022}}

@article{Dries_Speissegger_2000,
author = {van den Dries, Lou and Speissegger, Patrick},
title = {{The Field of Reals with Multisummable Series and the Exponential Function}},
journal = {Proceedings of the London Mathematical Society},
volume = {81},
number = {3},
pages = {513-565},
keywords = {o-minimal structures, model completeness, power series, blowing-up},
@doi = {https://doi.org/10.1112/S0024611500012648},
@url = {https://londmathsoc.onlinelibrary.wiley.com/doi/abs/10.1112/S0024611500012648},
eprint = {https://londmathsoc.onlinelibrary.wiley.com/doi/pdf/10.1112/S0024611500012648},
abstract = {We show that the field of real numbers with multisummable real power series is model complete, o-minimal and polynomially bounded. Further expansion by the exponential function yields again a model complete and o-minimal structure which is exponentially bounded, and in which the Gamma function on the positive real line is definable. 2000 Mathematics Subject Classification: primary 03C10, 32B05, 32B20; secondary, 26E05.},
year = {2000}
}

@article{InfiniteOMinimal,
author = {Rolin, J.-P. and Sanz, F. and Schäfke, R.},
title = {{Quasi-analytic Solutions of Analytic Ordinary Differential Equations and O-minimal Structures}},
journal = {Proceedings of the London Mathematical Society},
volume = {95},
number = {2},
pages = {413-442},
@doi = {https://doi.org/10.1112/plms/pdm016},
@url = {https://londmathsoc.onlinelibrary.wiley.com/doi/abs/10.1112/plms/pdm016},
eprint = {https://londmathsoc.onlinelibrary.wiley.com/doi/pdf/10.1112/plms/pdm016},
abstract = {Abstract It is well known that the non-spiraling leaves of real analytic foliations of codimension 1 all belong to the same o-minimal structure. Naturally, the question arises of whether the same statement is true for non-oscillating trajectories of real analytic vector fields. We show, under certain assumptions, that such a trajectory generates an o-minimal and model-complete structure together with the analytic functions. The proof uses the asymptotic theory of irregular singular ordinary differential equations in order to establish a quasi-analyticity result from which the main theorem follows. As applications, we present an infinite family of o-minimal structures such that any two of them do not admit a common extension, and we construct a non-oscillating trajectory of a real analytic vector field in ℝ5 that is not definable in any o-minimal extension of ℝ.},
year = {2007}
}

@article{Wilkies1996,
 ISSN = {08940347, 10886834},
 @URL = {http://www.jstor.org/stable/2152916},
 author = {A. J. Wilkie},
 journal = {Journal of the American Mathematical Society},
 number = {4},
 pages = {1051--1094},
 publisher = {American Mathematical Society},
 title = {{Model Completeness Results for Expansions of the Ordered Field of Real Numbers by Restricted Pfaffian Functions and the Exponential Function}},
 urldate = {2023-10-14},
 volume = {9},
 year = {1996}
}

\end{document}